\documentclass[a4paper,12pt]{amsart}
\usepackage[T1]{fontenc}
\usepackage[utf8]{inputenc}
\usepackage[margin=2.5cm]{geometry}

\usepackage{amsmath, amssymb,amsfonts}
\usepackage{array}
\usepackage{multirow}
\usepackage[dvipsnames]{xcolor}
\usepackage{float}
\usepackage{tikz}
\usetikzlibrary{shapes.misc}
\usetikzlibrary{shapes.geometric}
\usetikzlibrary {decorations.pathmorphing}

\tikzset{unseulcoin fill/.style={append after command={
   \pgfextra
        \draw[sharp corners, fill=#1, color=#1, line width = 0mm]% 
    (\tikzlastnode.west)% 
    [rounded corners=0pt] |- (\tikzlastnode.north)% 
    [rounded corners=0pt] -| (\tikzlastnode.east)% 
    [rounded corners=0pt] |- (\tikzlastnode.south)% 
    [rounded corners=5pt] -| (\tikzlastnode.west);
   \endpgfextra}}}

\tikzset{black node/.style = {draw=black,fill=black,circle, inner sep=0, minimum size=0.15cm}}
\usepackage{multicol}

\usepackage{hyperref}
\hypersetup{
  pdftitle = {Degeneracy and list coloring},
  pdfauthor = {L. Esperet},
  colorlinks = true,
  linkcolor = black!30!red,
  citecolor = black!30!green
}
  \usepackage[capitalise, compress, nameinlink, noabbrev]{cleveref}
\crefname{rule}{}{}
\creflabelformat{rule}{#2(R#1)#3}

\usepackage{paralist}

\newtheorem{theorem}{Theorem}[section]
\newtheorem{corollary}[theorem]{Corollary}
\newtheorem{definition}[theorem]{Definition}

\newtheorem{proposition}[theorem]{Proposition}
\newtheorem{lemma}[theorem]{Lemma}

\newtheorem{observation}[theorem]{Observation}

\newtheorem{claim}[theorem]{Claim}

\def\cqedsymbol{\ifmmode$\lrcorner$\else{\unskip\nobreak\hfil
\penalty50\hskip1em\null\nobreak\hfil$\lrcorner$
\parfillskip=0pt\finalhyphendemerits=0\endgraf}\fi}

\newcommand{\mad}{\operatorname{mad}}
\newcommand{\LLY}{\mathrm{LLY}}
\newcommand{\Lip}{\operatorname{Lip}}

\newcommand{\E}{\mathrm{Ext}_x}

\let\le\leqslant
\let\ge\geqslant
\let\leq\leqslant
\let\geq\geqslant

\title[Positive Lin--Lu--Yau curvature, planar graphs, and forbidden minors]{Positive Lin--Lu--Yau curvature,\\ planar graphs, and forbidden minors}

\author[L.~Esperet]{Louis Esperet}
\address[L.~Esperet]{Univ.\ Grenoble Alpes, CNRS, Laboratoire G-SCOP,
  Grenoble, France}
\email{louis.esperet@grenoble-inp.fr}

\author[J.~Semaan]{Julie Semaan}
\address[J.~Semaan]{Univ.\ Grenoble Alpes, Institut Fourier, France}
\email{julie.semaan@univ-grenoble-alpes.fr}

\thanks{The authors were partially supported by the French ANR Projects
  ENEDISC (ANR-24-CE48-7768) and Mimétique (ANR-25-CE48-4089), and by LabEx
  PERSYVAL-lab (ANR-11-LABX-0025).}

\date{\today}

\begin{document}

\begin{abstract}
Over the past few decades, several notions of graph curvature have been
introduced. Among them, Ollivier
curvature, defined through optimal transport, and its Lin--Lu--Yau
variant provide notions of curvature for graphs which retain several
features of their counterpart in Riemannian geometry. A classical problem in this context is to understand the global structure of positively curved spaces. 

Lu and Wang proved that there are only finitely many positively curved planar
graphs of minimum degree at least \(3\). 
In this paper, we study two directions motivated by this result.
We first show that for fixed \(2\leq s\leq t\), there are only finitely many positively curved
\(K_{s,t}\)-minor-free graphs with minimum degree at least \(s\). In particular, this generalizes the result of Lu and Wang to graphs embeddable on any fixed surface. We then give a complete classification of  positively curved \(5\)-connected plane triangulations. In particular, we show that there are only five examples, on
\(12,14,15,16\), and \(17\) vertices. 
\end{abstract}

\maketitle

\section{Introduction}

Curvature, in its various forms, may be viewed as a measure of deviation
from flatness. In Riemannian geometry, Ricci curvature measures, roughly
speaking, the average infinitesimal spreading of geodesics, the intrinsic
analogues of straight lines. Lower bounds on Ricci curvature have strong
global consequences; for instance, the Bonnet--Myers theorem asserts that
a uniform positive lower bound on the Ricci curvature yields an upper bound
on the diameter.

Although Ricci curvature is originally defined through the differential
structure of a manifold, Ollivier observed that it admits a remarkably
simple metric interpretation \cite{Ollivier2009}. On a positively Ricci curved manifold, small balls
centred at nearby points are, on average, closer to one another than their
centres are. By measuring this using optimal transport, and considering the
random walk given by the normalized volume measure on small balls, classical Ricci curvature is recovered, after rescaling, in
the second-order behaviour of the corresponding transportation distance.
This led Ollivier to introduce a notion of Ricci curvature depending
only on a metric and a random walk, which therefore extends naturally to
non-smooth spaces and, in particular, to graphs.

On an unweighted graph, a natural choice is the $p$-lazy simple random
walk: from a vertex, the walk stays there with probability $p$ and
otherwise moves to a uniformly chosen neighbour. The corresponding
Ollivier curvature thus gives a family of curvature notions depending on
the idleness parameter $p$. Lin, Lu and Yau \cite{LinLuYau2011}
introduced a parameter-free modification by considering the infinitesimal
behaviour of this family as $p$ tends to one. The resulting
Lin--Lu--Yau curvature retains several of the geometric features of Ricci
curvature; in particular, a uniform positive lower bound yields a
Bonnet--Myers type diameter bound.

Lin--Lu--Yau curvature is also closely related to the analytic structure of
the graph. Münch and Wojciechowski \cite{MunchWojciechowski2019} gave an
equivalent formulation in terms of the graph Laplacian, thereby connecting
curvature bounds with the associated heat flow. In particular, lower
curvature bounds can be characterized through the contraction of
Lipschitz functions under the heat semigroup. 

\medskip

The global structure of graphs with positive Lin--Lu--Yau curvature has been investigated extensively in the past few years, see for instance \cite{LiuLuWang2024,LuWang2026} for sparse graphs  and \cite{CLY26} for dense graphs, and the references therein.
The main motivation of our work comes from the following two recent results.

\begin{theorem}[\cite{LiuLuWang2024}]\label{thm:outerplanar}
  Every connected outerplanar graph of minimum
  degree at least $2$ and positive Lin--Lu--Yau curvature has order at most 10.
\end{theorem}

\begin{theorem}[\cite{LuWang2026}]\label{thm:planar}
  Every connected planar graph of minimum
  degree at least $3$ and positive Lin--Lu--Yau curvature has maximum degree at
  most 17 and order at most $17^{544}$.
\end{theorem}
The order in Theorem \ref{thm:outerplanar} is tight, and the
maximum-degree bound in Theorem \ref{thm:planar} is almost sharp:
Lu and Wang exhibit a graph attaining maximum degree
$16$ and conjecture that their upper bound of $17$
can be improved to $16$. On the other hand the order in Theorem \ref{thm:planar} is certainly very far from tight, the largest known
examples have order 26.

\medskip

To understand the minimum degree requirement, it is useful to
compute the Lin--Lu--Yau curvature of the complete bipartite graph
$K_{s,t}$, where $1\leq s\leq t$. By symmetry, the Lin--Lu--Yau curvature has the same value along every
edge, and a quick computation reveals that this value is equal to
$\frac{2}{t}>0$. It follows that that taking $s=2$ and arbitrarily large $t$
gives an infinite family of connected planar graphs of minimum degree $2$
and positive Lin--Lu--Yau curvature. Similarly, taking $s=1$ and arbitrarily large $t$
gives an infinite family of connected outerplanar graphs 
and positive Lin--Lu--Yau curvature and minimum degree $1$. This shows that the minimum degree
requirements in these results are necessary.

\medskip

This raises the following question: What graph property forces that there are only finitely many
connected graphs with positive Lin--Lu--Yau curvature? 
The results above show  that planar graphs with minimum degree at least 3 and  outerplanar graphs with minimum degree at least 2 have this property. Our first result is a far-reaching generalization of Theorems \ref{thm:planar} and \ref{thm:outerplanar}.

\begin{theorem}\label{thm:kstmain}
  For every integers $2\le s\le t$, there are only finitely many connected locally finite $K_{s,t}$-minor-free graph of minimum
  degree at least $s$ and positive Lin--Lu--Yau curvature.
\end{theorem}

Again, the infinite family of graphs $(K_{s-1,t})_{t\ge 1}$ shows that
the minimum degree assumption cannot be dropped. A natural question is
whether $K_{s,t}$ minors could be replaced by $K_{s,t}$ subgraphs in
the result. This is not the case for $s\ge 2$, as the hypercube $\{0,1\}^d$ has Lin--Lu--Yau curvature $\frac{2}{d}$ along every edge
\cite{cushing} and excludes
$K_{2,3}$ as a subgraph. We will
actually prove a slight strengthening of Theorem \ref{thm:kstmain},
where we only forbid $K_{s,t}$ minors obtained from contracting
disjoint stars (see Theorem \ref{thm:1-shallow}), which can be seen as
the $K_{s,t}$-minors that are the ``closest'' to $K_{s,t}$ subgraphs. In this setting we will need to add a bounded density condition, which is also necessary by the hypercube example.

\medskip

As outerplanar graphs exclude $K_{2,3}$ as a minor and planar graphs
exclude $K_{3,3}$ as a minor we obtain direct extensions of Theorems
\ref{thm:planar} and \ref{thm:outerplanar}. It directly follows from
Euler's formula that any graph embeddable on a fixed surface  excludes
$K_{3,t}$ as a minor, for some $t$ that only depends on the
surface. We therefore get a direct extension of Theorem \ref{thm:planar} to any
surface other that the sphere, with the same minimum degree
assumption. 

\medskip

This leads us to our second main result. The bounds we obtain on the number of vertices of our positively curved graphs are quite large, as the original bound of Lu and Wang on positively curved planar graphs of minimum degree at least 3 \cite{LuWang2026}. We would like to have instead a finer understanding of the positively curved objects under study, for instance a precise list of all these (finitely many) objects. This seems out of reach at the moment for general planar graphs of minimum degree 3. The original motivation of Lu and Wang \cite{LuWang2026} for studying planar graphs was that they can be seen as a graphic analogue of Riemannian surfaces. Here, we go one step further in this direction by restricting ourselves to 5-connected planar triangulations, which can be seen as the closest graphic analogue to smooth Riemannian planes. These additional properties allow for a complete description of the positively curved graphs in the class, and we hope the ideas introduced in this part might help to extend the result beyond triangulations and with less stringent connectivity properties.

\begin{theorem}\label{thm:5ct}
    There are only five 5-connected plane triangulations with positive Lin--Lu--Yau curvature, depicted in Figures \ref{fig:G12}, \ref{fig:G14}, \ref{fig:G15}, \ref{fig:G16}, and \ref{fig:G17}. Each has maximum degree at most 6 and order at most 17.
\end{theorem}

\subsection*{Mention of very recent related work} While we were polishing our manuscript, two papers related to the present work were posted on arXiv. In \cite{BLS26}, the authors prove the case $s=2$ of our Theorem \ref{thm:kstmain}, with an explicit (and small) dependence of the degree and order in $t$. In \cite{Song26}, the author proves the corresponding result for $K_t$-minor-free graphs of minimum degree at least $t-1$. As $K_t$-minor-free graphs are $K_{t-1,t-1}$-minor-free, the non-quantitative part of her result is a direct consequence of Theorem \ref{thm:kstmain} with $s=t-1$.

\subsection*{Organization of the paper} We start with some basic graph definitions in Section \ref{sec:preliminaries}, and we introduce the Lin--Lu--Yau curvature in Section \ref{sec:LLY}. After mentioning some useful results on graph minors and extremal combinatorics in Section \ref{sec:prelgraph}, we prove Theorem \ref{thm:kstmain} in Section \ref{sec:kst}. Section \ref{sec:prel5ct} is devoted to some observations and preliminary results on 5-connected planar triangulations, and Section \ref{sec:5ct} contains the proof of Theorem \ref{thm:5ct}.

\section{Graph-theoretic preliminaries}
\label{sec:preliminaries}

Throughout this paper, a graph is a pair $G=(V(G),E(G))$, where $V(G)$ is a finite or countably infinite set of vertices and $E(G)$ is its edge set. All graphs are assumed to be simple, undirected (unless otherwise stated), and locally finite, meaning that every vertex has finitely many neighbours.

\subsection{Basic graph notation}

Let $G$ be a graph. For a vertex $v\in V(G)$, its neighbourhood is denoted by $N(v):=\{u\in V(G):uv\in E(G)\}$, and its closed neighbourhood by $N[v]:=N(v)\cup\{v\}$. The degree of $v$ is $d(v):=|N(v)|$.
We write $\delta(G):=\inf_{v\in V(G)}d(v)$ and $\Delta(G):=\sup_{v\in V(G)}d(v)$ for the minimum and maximum degree parameters of $G$, respectively. We allow $\Delta(G)=\infty$.

\smallskip

For a set $S\subseteq V(G)$, we write
  $N(S):=\bigcup_{v\in S}N(v)$. We emphasize that this is a somewhat
  unusual  convention, and $N(S)$ need not
  be disjoint from $S$.

\smallskip

The subgraph of $G$ induced by some subset $S\subseteq V(G)$ is denoted by $G[S]$.
In particular, $G[N(v)]$ and $G[N[v]]$ denote the subgraphs induced by
the open and closed neighbourhoods of $v$, respectively.

\smallskip

For vertices $u,v\in V(G)$, we denote their graph distance by
$d_G(u,v)$.

\smallskip

For $n\in\mathbb{N}$, we write $[n]:=\{1,\ldots,n\}$. If $X$ is a
finite set and $k\geq 0$ is an integer, then $\binom{X}{k}$ denotes
the family of all $k$-element subsets of $X$.
%We denote by $K_r$ the complete graph on $r$ vertices and by $K_{s,t}$ the complete bipartite graph with parts of sizes $s$ and $t$.

\subsection{Maximum average degree and graph minors}

The \emph{average degree} of a finite graph $H$ is the average of its vertex
degrees, or equivalently $2|E(H)|/|V(H)$ if $H$ is non-empty, and 0 otherwise. 
The 
\emph{maximum average degree} of a locally finite graph $G$,
denoted by $\mad(G)$, is the supremum of the average degrees its
finite subgraphs.
Thus $\mad(G)\leq d$ if and only if every finite subgraph $H$ of $G$ satisfies $2|E(H)|\leq d|V(H)|$.

\medskip
Let $G$ be a graph and let $H$ be a finite graph. We say that $H$ is a \emph{minor} of $G$, and write $H\preccurlyeq G$, if $H$ can be obtained from a finite subgraph of $G$ by a finite sequence of vertex deletions, edge deletions, and edge contractions.
Equivalently, $H\preccurlyeq G$ if there exists a family $(B_v)_{v\in V(H)}$ of nonempty, pairwise disjoint subsets of $V(G)$ such that $G[B_v]$ is connected for every $v\in V(H)$ and, whenever $uv\in E(H)$, there is an edge of $G$ with one endpoint in $B_u$ and the other in $B_v$. Such a family is called a \emph{minor model} of $H$ in $G$, and the sets $B_v$ are called its \emph{branch sets}.
If $H$ is not a minor of $G$, we say that $G$ is $H$-minor-free and write $H\not\preccurlyeq G$.

\medskip

We shall use the following theorem of Mader~\cite{Mader1968}.

\begin{theorem}[\cite{Mader1968}]
\label{thm:mader}
For every graph $H$, there exists a constant $c_H$ such that every
finite $H$-minor-free graph has average degree at most $c_H$.
\end{theorem}

We shall also use the following classical theorem of Tur\'an~\cite{turan1941extremal}.

\begin{theorem}[\cite{turan1941extremal}]\label{thm:turan}
Let $r\geq 1$ be an integer. If a graph $G$ on $n$ vertices does not contain $K_{r+1}$ as a subgraph, then
\[
|E(G)|\leq \left(1-\frac{1}{r}\right)\frac{n^2}{2}.
\]
\end{theorem}
\section{Lin--Lu--Yau curvature}\label{sec:LLY}
Throughout this section, let $G$ be a simple, connected,
locally finite graph with at least two vertices.
Before introducing the Lin--Lu--Yau curvature~\cite{LinLuYau2011}, we shall recall a few definitions.

\smallskip

A \emph{probability measure} on $V(G)$ is a function $\mu\colon V(G)\to[0,1]$ satisfying
$
\sum_{z\in V(G)}\mu(z)=1.
$
Let $\mu$ and $\nu$ be finitely supported probability measures on $V(G)$. A \emph{coupling}, or \emph{transport plan}, of $\mu$ and $\nu$ is a function
$
\pi\colon V(G)\times V(G)\to[0,\infty)
$
such that
$
\sum_{v\in V(G)}\pi(u,v)=\mu(u)
$
for every $u\in V(G)$ and
$
\sum_{u\in V(G)}\pi(u,v)=\nu(v)
$
for every $v\in V(G)$. We denote by $\Pi(\mu,\nu)$ the set of all couplings of $\mu$ and $\nu$.

\medskip

Let $\mu$ and $\nu$ be finitely supported probability measures on $V(G)$. Their \emph{$1$-Wasserstein distance} is
\begin{equation}
\label{eq:wasserstein-distance}
W_1(\mu,\nu)
:=
\inf_{\pi\in\Pi(\mu,\nu)}
\sum_{u,v\in V(G)}d_G(u,v)\pi(u,v).
\end{equation}

Thus $W_1(\mu,\nu)$ is the minimum transportation cost for moving $\mu$ to $\nu$ when transporting one unit of mass from $u$ to $v$ has cost $d_G(u,v)$.

\begin{definition}[Lazy random walk measure]
Let $p\in[0,1]$ and $x\in V(G)$. The $p$-lazy random walk measure based at $x$ is the probability measure $\mu_x^p$ given by
\begin{equation}
\label{eq:lazy-random-walk}
\mu_x^p(z)=
\begin{cases}
p, & \text{ if } z=x,\\
\dfrac{1-p}{d(x)}, & \text{ if } z\sim x,\\
0, & \text{otherwise}.
\end{cases}
\end{equation}
\end{definition}

In other words, the walk remains at $x$ with probability $p$ and moves to a uniformly chosen neighbour of $x$ with probability $1-p$.

\begin{definition}[Ollivier and Lin--Lu--Yau curvature]
Let $x,y\in V(G)$ be distinct vertices and let $p\in[0,1)$. The $p$-Ollivier curvature of $x$ and $y$ is
\begin{equation}
\label{eq:p-ollivier-curvature}
\kappa_p(x,y)
:=
1-\frac{W_1(\mu_x^p,\mu_y^p)}{d_G(x,y)}.
\end{equation}
The Lin--Lu--Yau curvature (abbreviated as LLY curvature) of $x$ and $y$ is
\begin{equation}
\label{eq:lly-curvature-definition}
\kappa_{\LLY}(x,y)
:=
\lim_{p\to1^-}\frac{\kappa_p(x,y)}{1-p}.
\end{equation}
\end{definition}

The first definition is Ollivier's coarse Ricci curvature applied to
the lazy random walk measures~\cite{Ollivier2009}, while the second is
its Lin--Lu--Yau modification~\cite{LinLuYau2011}.

\smallskip

 It was proved in~\cite{LinLuYau2011} that if
$\kappa_{\LLY}(x,y)>0$ for every edge $xy$, then
$\kappa_{\LLY}(x,y)>0$ for every pair $x,y$ of vertices. So for the
purpose of studying positive curvature, we can thus restrict our
attention to curvature along edges. In particular, if $xy\in E(G)$, then $d_G(x,y)=1$ and
$
\kappa_p(x,y)=1-W_1(\mu_x^p,\mu_y^p).
$

\medskip

We say that $G$ has 
\emph{positive LLY curvature} if $\kappa_{\LLY}(x,y)>0$ for every edge
$xy\in E(G)$. A graph with positive LLY curvature is said to be
\emph{positively LLY curved}.

\medskip

For our purposes, it will be convenient to use the following
limit-free characterization of LLY curvature.

\begin{definition}[Normalized graph Laplacian]
For a function $f\colon V(G)\to\mathbb{R}$, the normalized graph Laplacian is defined by
\begin{equation}
\label{eq:normalized-laplacian}
\Delta f(v)
:=
\frac{1}{d(v)}
\sum_{w\sim v}\bigl(f(w)-f(v)\bigr).
\end{equation}
\end{definition}

\begin{definition}[Lipschitz functions]
A function $f\colon V(G)\to\mathbb{R}$ is called $1$-Lipschitz if
$
|f(u)-f(v)|\leq d_G(u,v)
$
for all $u,v\in V(G)$. We denote by $\Lip(1)$ the set of all $1$-Lipschitz functions on $V(G)$.
\end{definition}

M\"unch and Wojciechowski~\cite{MunchWojciechowski2019} proved the following characterization of Lin--Lu--Yau curvature.

\begin{theorem}[\cite{MunchWojciechowski2019}]
\label{thm:limit-free}
Let $xy\in E(G)$. Then
\begin{equation}
\label{eq:limit-free-curvature}
\kappa_{\LLY}(x,y)
=
\inf_{\substack{f\in\Lip(1)\\ f(y)-f(x)=1}}
\bigl(\Delta f(x)-\Delta f(y)\bigr).
\end{equation}
\end{theorem}

This limit-free characterization is the form of LLY curvature used
throughout the proof.

\section{Some preliminary results }
\label{sec:prelgraph}

The proof of Theorem~\ref{thm:maximum-degree} repeatedly uses two
principles. The first is geometric: positive curvature along an edge
forces a substantial subset of the neighbourhood of one endpoint to
have neighbours in a correspondingly large part of the neighbourhood
of the other, up to an explicit error term. We will use the following
form of this observation, which is due to Liu, Lu and Wang \cite[Lemma~1]{LuWang2026}.

\begin{lemma}{\cite[Lemma~1]{LuWang2026}}
\label{lem:curvature-expansion}
Let $G$ be a simple connected locally finite graph, let $xy\in E(G)$,
and let $S\subseteq N(y)\setminus\{x\}$.
Put $\sigma:=|S|$, $k:=|S\cap N(x)|$, and
$c:=|N(x)\cap N(y)|$. If $\kappa_{\LLY}(x,y)>0$, then
\begin{equation}
\label{eq:curvature-expansion}
|N(S)\cap N(x)|
>
\frac{\sigma}{d(y)}d(x)
-(k+1+c)
+|N(S)\cap N(x)\cap N(y)|.
\end{equation}
In particular,
$$
|N(S)\cap N(x)|
>
\frac{|S|}{d(y)}d(x)-(k+1+c).
$$
\end{lemma}

The second principle we use has an extremal combinatorics flavour: a
sufficiently large family of large subsets of a common finite set
contains a prescribed number of sets with a large common
intersection. This result will be used in order to construct
complete bipartite minors. It is obtained using the classical
K\H{o}v\'ari--S\'os--Tur\'an counting argument \cite{KovariSosTuran1954}. We provide a proof in Appendix \ref{sec:app} for completeness. 

\begin{lemma}
\label{lem:common-intersection}
Let $r,q,N\in\mathbb{N}$, and let $S_1,\ldots,S_q\subseteq[N]$. Suppose that
$|S_i|\geq m$ for every $i\in[q]$, where $m>0$. If
$q\geq rN/m$, then there exist distinct indices
$i_1,\ldots,i_r\in[q]$ such that
\begin{equation*}
\label{eq:common-intersection-bound}
|S_{i_1}\cap\cdots\cap S_{i_r}|
\geq N\left(\frac{m}{eN}\right)^r,
\end{equation*}
where $e\approx 2.718$ is the base of the natural logarithm.
\end{lemma}

\section{The order of $K_{s,t}$-minor-free graphs with positive LLY curvature}\label{sec:kst}

In the remainder, we write $\mathcal{H}_{s,t}$ for the
class of all connected locally finite $K_{s,t}$-minor-free graphs with minimum degree at least $s$
and positive LLY curvature. 
The goal of this section is to prove Theorem \ref{thm:kstmain}, which
states that $\mathcal{H}_{s,t}$ contains only finitely many graphs. For this it will be enough to prove that such
graphs have uniformly bounded maximum degree.

\begin{theorem}
\label{thm:maximum-degree}
For all integers $2\leq s\leq t$, there exists an integer
$\Delta_{s,t}\geq 2$ such that every graph in $\mathcal{H}_{s,t}$ has maximum degree at most $\Delta_{s,t}$.
\end{theorem}

We now explain how to deduce Theorem \ref{thm:kstmain} from Theorem
\ref{thm:maximum-degree} (the argument is similar to the argument used
in the main result of \cite{LuWang2026} on planar graphs).

\medskip

\noindent \textit{Proof of Theorem \ref{thm:kstmain} assuming Theorem
  \ref{thm:maximum-degree}.}
Let $G$ be a connected locally finite $K_{s,t}$-minor-free graphs of minimum degree at least $s$ and
positive LLY curvature. By Theorem~\ref{thm:maximum-degree}, $\Delta(G)\le \Delta_{s,t}$. It can be observed that any graph $H$ of positive
  curvature satisfies $\kappa_{\LLY}(H)\ge
  \tfrac1{\Delta(H)(\Delta(H)-1)}$ (see also \cite[Lemma
  2]{LuWang2026}). Hence, $\kappa_{\LLY}(G)> 1/\Delta_{s,t}^2$. Lin,
  Lu, and Yau \cite{LinLuYau2011} proved that any connected graph $H$ of maximum
  degree at most $\Delta$ with LLY curvature
  $\kappa_{\LLY}(H)\ge \kappa>0$ has order at most
  $\Delta^{1+2/\kappa}$. It follows that $G$ has order at most
  $f(\Delta_{s,t})$, with $f(x)=x^{2x^2+1}$. 

  As observed in \cite{BLS26,Song26}, a significantly more economic bound can be obtained by combining spectral estimates of Lin, Lu, and Yau \cite[Theorem 4.2]{LinLuYau2011} with bounds on normalized eigenvalues of minor-free graphs. For instance, using \cite[Corollary 1.14]{KS26}, one can obtain that $G$ has order a most $f(\Delta_{s,t})$ with $f(x)=O(x^3 (t\log t)^2)$.
  \hfill $\Box$

  \medskip

The proof of Theorem~\ref{thm:maximum-degree} is divided into several steps, each developed in a separate
subsection. We first give the degree estimates needed for the
argument, before combining them in the final subsection to prove
Theorem~\ref{thm:maximum-degree}.

\subsection{Large degrees around a vertex of sufficiently large degree}

In this subsection, we show that a vertex $x$ of sufficiently large degree
in a graph $G\in \mathcal{H}_{s,t}$ must have most of its
neighbours of large degree as well. More precisely, for every fixed
integer $K\geq s$, the number of neighbours of degree at most $K$ of $x$
is bounded by a constant depending only on $s,t$ and $K$,
independently of the degree of the vertex $x$. The following proposition
formalises this fact:
\begin{proposition}
\label{prop:low-degree}
 For every integers $ 2\leq s\leq t$ integers and $ K\geq s$, there
 exists $\Delta_0=\Delta_0(s,t,K)\in\mathbb{N}$ such that, for every graph $G\in\mathcal{H}_{s,t}$ and every vertex $x\in V(G)$ with $d(x)\geq\Delta_0$, we have
$$
\left|\{u\in N(x):d(u)\leq K\}\right|
\leq(t-1)\binom{2(s-1)K}{s-1}.
$$
\end{proposition}

\begin{proof}
Let $2\leq s\leq t$ and $K\geq s$ be fixed integers. Let $\Delta_0
=\max\{(2K(K+1), \, t(2eK)^{s-1}\}$. Let $G\in \mathcal{H}_{s,t}$ and
let $x$ be a vertex of $G$ with $d(x):=D \geq \Delta_0$. We set
$$
U_K:=\{u\in N(x):d(u)\leq K\},
\qquad
Z:=\bigcup_{u\in U_K}\bigl(N(u)\setminus\{x\}\bigr).
$$ 

Fix $z\in Z$. By definition, $z\neq x$, and there exists $u\in U_K$ such that $z\sim u$. We may apply Lemma~\ref{lem:curvature-expansion} to the edge $xu$ with $S=\{z\}$. Since $|S\cap N(x)|\leq1$ and $|N(x)\cap N(u)|\leq d(u)-1\leq K-1$, the lemma gives
\begin{align*}
|N(z)\cap N(x)|
&>\frac{D}{d(u)}-\bigl(|S\cap N(x)|+1+|N(x)\cap N(u)|\bigr)\\
&\geq\frac{D}{d(u)}-(K+1)\\
&\geq\frac{D}{K}-(K+1).
\end{align*}
Thus, as $D\geq \Delta_0 \geq 2K(K+1)$, we have
\begin{equation}
\label{eq:z-large-neighbourhood}
|N(z)\cap N(x)|>\frac{D}{2K}.
\end{equation}
For each $z\in Z$, set $S_z:=N(z)\cap N(x)$.
We first claim that $|Z|<2(s-1)K$. Suppose, for contradiction, that $|Z|\geq2(s-1)K$. We claim that  this allows us to find distinct vertices $z_1,\ldots,z_{s-1}\in Z$ such that $|S_{z_1}\cap\cdots\cap S_{z_{s-1}}|\geq t$, and hence to construct a $K_{s,t}$ subgraph in $G$, which gives a contradiction.
To see this, note that the family $(S_z)_{z\in Z}$ consists of subsets of the $D$-element set $N(x)$, each of cardinality greater than $D/(2K)$. Since
$$
2(s-1)K=\frac{(s-1)D}{D/(2K)},
$$
Lemma~\ref{lem:common-intersection} implies that there exist distinct vertices $z_1,\ldots,z_{s-1}\in Z$ such that
$$
|S_{z_1}\cap\cdots\cap S_{z_{s-1}}|
\geq D\left(\frac{1}{2eK}\right)^{s-1}.
$$
As $D\geq t(2eK)^{s-1}$, this intersection contains at least $t$ vertices. Choose distinct vertices
$$
v_1,\ldots,v_t\in S_{z_1}\cap\cdots\cap S_{z_{s-1}}.
$$
Each $v_j$ is adjacent to $x$ and to every $z_i$. Moreover, the sets $\{x,z_1,\ldots,z_{s-1}\}$ and $\{v_1,\ldots,v_t\}$ are disjoint: $x\notin N(x)$, while $z_i=v_j$ would imply $z_i\in S_{z_i}\subseteq N(z_i)$, which cannot happen in a simple graph. These vertices therefore give a copy of $K_{s,t}$ as a subgraph, contradicting $K_{s,t}\not\preccurlyeq G$. Hence
\begin{equation}
\label{eq:Z-bound}
|Z|<2(s-1)K.
\end{equation}
It remains to bound the cardinality of the set $U_K.$
Since $\delta(G)\geq s$, every $u\in U_K$ has degree at least $s$, and therefore at least $s-1$ neighbours other than $x$. For each $u\in U_K$, choose a set $R_u\subseteq N(u)\setminus\{x\}$ with $|R_u|=s-1$. By definition of $Z$, we have $R_u\subseteq Z$.

We claim that, for every subset $R $ of $Z$, of cardinality $s-1$, there are at most $t-1$ vertices $u\in U_K$ such that $R_u=R$. 
Otherwise, there would exist distinct $u_1,\ldots,u_t\in U_K$ with $R_{u_1}=\cdots=R_{u_t}=R$ and this will give rise to a copy of $K_{s,t}$ as a subgraph. Indeed, every $u_j$ is adjacent to $x$ and to every vertex of $R$. Moreover, $x\notin R$ and $u_j\notin R$ for every $j$, since $R=R_{u_j}\subseteq N(u_j)\setminus\{x\}$ and $G$ is simple. Thus the sets $\{x\}\cup R$ and $\{u_1,\ldots,u_t\}$ would give a copy of $K_{s,t}$ as a subgraph, again a contradiction.

Finally,
$$
U_K=\bigcup_{R\in\binom{Z}{s-1}}\{u\in U_K:R_u=R\},
$$
and therefore
\begin{align*}
|U_K|
&\leq\sum_{R\in\binom{Z}{s-1}}\left|\{u\in U_K:R_u=R\}\right|\\
&\leq(t-1)\binom{|Z|}{s-1}\\
&\leq(t-1)\binom{2(s-1)K}{s-1},
\end{align*}
where the last inequality follows from \eqref{eq:Z-bound}.
\end{proof}

As the bound in the previous proposition does not depend on $d(x)$,
the proportion of neighbors of $x$ of degree at most $K$ tends to $0$
as $d(x)$ tends to infinity. We therefore obtain the following
immediate corollary:

\begin{corollary}
\label{cor:most-neighbours-large}
Let $2\leq s\leq t$ be integers, let $K\geq s$ be an integer, and let
$\varepsilon\in(0,1)$. There exists
$\Delta_1=\Delta_1(s,t,K,\varepsilon)\in\mathbb{N}$ such that, for
every graph $G\in \mathcal{H}_{s,t}$, and
every vertex $x$ with $d(x)\ge \Delta_1$, we have
$$
\left|\{u\in N(x):d(u)>K\}\right|\geq(1-\varepsilon)d(x).
$$
\end{corollary}

%\begin{proof}Let $2\leq s\leq t$ and $K\geq s$ be fixed integers and let $\varepsilon \in (0,1)$. 
%By Proposition~\ref{prop:low-degree}, if $\Delta \geq \Delta_0$, then:
%$$
%\left|\{u\in N(x):d(u)\leq K\}\right|
%\leq(t-1)\binom{2(s-1)K}{s-1}.
%$$
%The right-hand side depends only on $s,t$, and $K$. Taking $\Delta\geq \max\bigg(\Delta_0,\frac{(t-1)}{\varepsilon }\binom{2(s-1)K}{s-1} \bigg)=:\Delta_1$, gives the conclusion.
%\end{proof}

\subsection{Bounded internal degree under bounded maximum average degree}

We next observe a consequence of having a bounded maximum average degree: for any vertex $x$, only a small proportion of the neighbors of $x$ can have many neighbors within $N(x)$. We record this fact in the following proposition:

\begin{proposition}
\label{prop:internal-degree}
Let $d>0$, let $G$ be a graph with $\mad(G)\leq d$, and let $x\in V(G)$. For each $u\in N(x)$, set $c_u:=|N(u)\cap N(x)|$. Then, for every $C>0$,
$$
\left|\{u\in N(x):c_u>C\}\right|\leq\frac{d\cdot (d(x)+1)}{C}.
$$
\end{proposition}

\begin{proof}
 Let $x\in V(G)$ of degree $D:=d(x)$ and consider the induced subgraph $H:=G[N[x]]$. Since $H$ has $D+1$ vertices and $\mad(G)\leq d$, we have $2|E(H)|\leq d(D+1)$. Observe that the edges of $H$ consist of the $D$ edges joining $x$ to its neighbours together with the edges of $G[N(x)]$, between the neighbors of $x$. Hence
$$
2|E(G[N(x)])|\leq d(D+1)-2\Delta\leq d(D+1).
$$
Moreover, for each $u\in N(x)$, the quantity $c_u$ is precisely the degree of $u$ in $G[N(x)]$. Thus,
$$
\sum_{u\in N(x)}c_u=2|E(G[N(x)])|\leq d(D+1).
$$
Now let $X_C:=\{u\in N(x):c_u>C\}$. As $c_u>C$ for every $u\in X_C$,
$$
\sum_{u\in X_C}c_u>C\sum_{u\in X_C}1=C|X_C|.
$$
On the other hand, $\sum_{u\in X_C}c_u\leq\sum_{u\in N(x)}c_u\leq d(D+1)$. Therefore $C|X_C|<d(D+1)$, which proves the claim.
\end{proof}

By choosing $C$ to be sufficiently large, the proportion of neighbours of $x$ having more than $C$ common neighbors with $x$ can be made arbitrarily small. We therefore obtain the following corollary:

\begin{corollary}
\label{cor:most-neighbours-bounded-internal-degree}
Let $d>0$ and $\varepsilon\in(0,1)$. There exists
$C=C(d,\varepsilon)>0$ such that, for every graph $G$ with
$\mad(G)\leq d$ and every vertex $x$,% of degree $\Delta$,
$$
\left|\left\{u\in N(x):|N(u)\cap N(x)|\leq C\right\}\right|\geq(1-\varepsilon)d(x).
$$
\end{corollary}

%\begin{proof}
%If $\Delta=0$, there is nothing to prove. Assume $\Delta\geq1$ and take $C:=2d/\varepsilon$. Proposition~\ref{prop:internal-degree} implies that
%$$
%\left|\{u\in N(x):c_u>C\}\right|\leq\frac{d(\Delta+1)}{C}\leq\frac{2d\Delta}{C}=\varepsilon \Delta.
%$$
%The result follows.
%\end{proof}

\subsection{Proof of Theorem \ref{thm:maximum-degree} }\label{sec:proof-main}

We first outline the proof: We shall proceed by contradiction, constructing in each case a (forbidden) $K_{s,t}$ minor. We suppose that $G$ contains a vertex
$x$ of sufficiently large degree. The preceding corollaries provide many neighbours $u$ of $x$ with large degree but only
boundedly many common neighbours with $x$. Consequently, their
external neighbourhoods $\E(u):=N(u)\setminus N[x]$ are large.
The case $s=2$ follows directly from Lemma \ref{lem:common-intersection}. For $s\geq3$, we distinguish two cases according to the
intersections of these external neighbourhoods.

If there are $s-1$ external neighbourhoods with sufficiently
small pairwise intersections, we remove their overlaps to obtain
pairwise disjoint sets, each retaining a sufficiently large
proportion of its original neighbourhood. Positive curvature
then ensures that many vertices of $N(x)$ have a neighbour in
each of these sets. This allows us to construct a $K_{s,t}$ minor
by contraction.

Otherwise, we use an auxiliary graph and Tur\'an's theorem
to find one external neighbourhood having substantial overlap
with many others. The common-intersection lemma implicit in the K\H{o}v\'ari--S\'os--Tur\'an theorem then yields
$s$ vertices with at least $t$ common neighbours, giving a
$K_{s,t}$ subgraph.

In either case, we contradict the exclusion of a $K_{s,t}$ minor.
This establishes the required degree bound.
We shall now proceed with the proof of Theorem \ref{thm:maximum-degree}.

\medskip

\noindent \textit{Proof of Theorem \ref{thm:maximum-degree}.} Let
$2\leq s\leq t$ be fixed integers and suppose for contradiction that for arbitrarily large integer
$D$, there exists a graph $G\in \mathcal{H}_{s,t}$ and a vertex $x\in
V(G)$ with $d(x)=\Delta$. Fix such a graph $G$ and vertex $x$. 
For every $u\in N(x)$, set $c_u:=|N(u)\cap N(x)|$ and $\E(u)
:=N(u)\setminus N[x]$. Thus $|\E(u)|=d(u)-1-c_u$.

\smallskip

By Theorem~\ref{thm:mader}, there exists a constant $d=d_{s,t}$
depending only on $s$ and $t$ such that  $\mad(G)\leq d$. All
constants chosen below therefore depend only on $s$ and $t$. Since $D$
may be taken arbitrarily large, we take it sufficiently large to
satisfy all the conditions below.

\medskip

%\noindent \textbf{The case $s = 2$.}

%\smallskip

\noindent \textbf{Assume first that $s=2$.}
By Corollaries \ref{cor:most-neighbours-large}
and~\ref{cor:most-neighbours-bounded-internal-degree} we may choose
$C$ and $K$ (depending only on $s$ and $t$), such
that for sufficiently large $D$,
%Set $C:=8d$. By Proposition~\ref{prop:internal-degree},
%\begin{align*}
%\left|\{u\in N(x):c_u>C\}\right|
%&\leq\frac{d\cdot(D+1)}{C}
%\leq\frac{D}{4}.
%\end{align*}
%Choose an integer $K\geq2$ sufficiently large that $K+1\geq2(C+1)$. Once $K$ is fixed, Corollary~\ref{cor:most-neighbours-large} with $\varepsilon=1/4$ gives, for all sufficiently large %$D$,
%\begin{align*}
%\left|\{u\in N(x):d(u)\leq K\}\right|
%\leq\frac{D}{4}.
%\end{align*}
\[U:=\{u\in N(x):d(u)>K\text{ and }c_u\leq C\}\] satisfies $|U|\geq
D/2$. In particular, $U\neq\varnothing$ for $D\ge 1$.

\medskip
Choose any vertex $u\in U$ and apply Lemma~\ref{lem:curvature-expansion} to the edge $xu$ with $S=\E(u)$. Since $\E(u)\cap N(x)=\varnothing$ and $|N(x)\cap N(u)|=c_u$, we obtain
\begin{align*}
|N(\E(u))\cap N(x)|
&>\frac{|\E(u)|}{d(u)}D-(1+c_u)\\
&=\left(1-\frac{1+c_u}{d(u)}\right)D-(1+c_u)\\
&\geq\left(1-\frac{C+1}{K+1}\right)D-(C+1)\\
&\geq\frac{D}{2}-(C+1).
\end{align*}
For $D$ sufficiently large, the last quantity is greater than $t+1$. We may therefore choose distinct vertices
\[
v_1,\ldots,v_t\in
\bigl(N(\E(u))\cap N(x)\bigr)\setminus\{u\}.
\]

Let $B:=\{u\}\cup \E(u)$. The set $B$ is connected because every vertex of $\E(u)$ is adjacent to $u$. Moreover, the sets
\[
\{x\},\quad B,\quad\{v_1\},\ldots,\{v_t\}
\]
are pairwise disjoint, since $\E(u)\cap N[x]=\varnothing$ and the
vertices $v_1,\ldots,v_t$ are distinct from $u$. Every $v_j$ is
adjacent to $x$, and since $v_j\in N(\E(u))$, it has a neighbour in
$E_u\subseteq B$. Hence $\{x\},B$ and $\{v_1\},\ldots,\{v_t\}$ form
the branch sets of a $K_{2,t}$ minor, contradicting
$K_{2,t}\not\preccurlyeq G$. Thus the theorem holds when $s=2$.

\medskip

\noindent \textbf{Assume now that $s\geq3$}.
By Corollaries \ref{cor:most-neighbours-large}
and~\ref{cor:most-neighbours-bounded-internal-degree} we may choose again
$C$ and $K$  (depending only on $s$ and $t$), such
that for sufficiently large $D$,
\[U:=\{u\in N(x):d(u)>K\text{ and }c_u\leq C\}\] satisfies $|U|\geq
D/2$. Here, after the
choice of $C$,  we choose
$K$ with the following two additional properties:
\begin{equation}\label{eq:bndKC1}
  (C+1)/(K+1)\leq1/(2s(s-1))
\end{equation}
\begin{equation}\label{eq:bndKC2}
(K-C)\left(\frac{1}{es(s-2)}\right)^{s-1}\geq t
\end{equation}

We now distinguish two cases according to the overlap of the external neighbourhoods $\E(u)$, $u\in U$.

\medskip

\noindent\textbf{Case 1: almost-disjoint external neighbourhoods.}

\smallskip

\noindent Suppose that there exist distinct vertices $u_1,\ldots,u_{s-1}\in U$ such that
\begin{equation}
\label{eq:small-pairwise-intersections}
|\E(u_i)\cap \E(u_j)|
\leq\frac{1}{s(s-2)}\min\{d(u_i),d(u_j)\}
\end{equation}
for every distinct $i,j\in\{1,\ldots,s-1\}$.

\smallskip

For each $i\in\{1,\ldots,s-1\}$, define
\[
S_i:=\E(u_i)\setminus\bigcup_{j\neq i}\E(u_j).
\]
The sets $S_1,\ldots,S_{s-1}$ are pairwise disjoint. Moreover,
\begin{align*}
|S_i|
&\geq|\E (u_i)|-\sum_{j\neq i}|\E(u_i)\cap \E(u_j)|\\
&\geq d(u_i)-1-c_{u_i}-\frac{s-2}{s(s-2)}\cdot d(u_i)\\
&\geq\frac{s-1}{s}\cdot d(u_i)-(C+1).
\end{align*}

Apply Lemma~\ref{lem:curvature-expansion} to the edge $xu_i$ with $S=S_i$. Since $S_i\subseteq \E(u_i)$, we have $S_i\cap N(x)=\varnothing$, while $|N(x)\cap N(u_i)|=c_{u_i}$. Hence
\begin{align*}
|N(S_i)\cap N(x)|
&>\frac{|S_i|}{d(u_i)}\cdot D-(1+c_{u_i})\\
&\geq\left(\frac{s-1}{s}-\frac{C+1}{d(u_i)}\right)D-(C+1)\\
&\geq\left(\frac{s-1}{s}-\frac{C+1}{K+1}\right)D-(C+1).
\end{align*}
By \eqref{eq:bndKC1}, $(C+1)/(K+1)\leq1/(2s(s-1))$, and therefore
\begin{equation}
\label{eq:expanded-Si}
|N(S_i)\cap N(x)|>
\left(\frac{s-1}{s}-\frac{1}{2s(s-1)}\right)D-(C+1).
\end{equation}

Since all the sets $N(S_i)\cap N(x)$ are contained in the $D$-element set $N(x)$,
\[
\left|\bigcap_{i=1}^{s-1}\bigl(N(S_i)\cap N(x)\bigr)\right|
\geq\sum_{i=1}^{s-1}|N(S_i)\cap N(x)|-(s-2)D.
\]
Using \eqref{eq:expanded-Si}, we obtain
\begin{align*}
\left|\bigcap_{i=1}^{s-1}\bigl(N(S_i)\cap N(x)\bigr)\right|
&>(s-1)\left(\frac{s-1}{s}-\frac{1}{2s(s-1)}\right)D\\
&\qquad-(s-1)(C+1)-(s-2)D\\
&=\frac{D}{2s}-(s-1)(C+1).
\end{align*}
For $D$ sufficiently large, the last quantity is greater than $t+s-1$. We may therefore choose distinct vertices
\[
v_1,\ldots,v_t\in
\bigcap_{i=1}^{s-1}\bigl(N(S_i)\cap N(x)\bigr)
\setminus\{u_1,\ldots,u_{s-1}\}.
\]

For each $i$, let $B_i:=\{u_i\}\cup S_i$. Each $B_i$ is connected because every vertex of $S_i$ is adjacent to $u_i$. The sets
\[
\{x\},\quad B_1,\ldots,B_{s-1},\quad\{v_1\},\ldots,\{v_t\}
\]
are pairwise disjoint. Indeed, the sets $S_i$ are pairwise disjoint and lie outside $N[x]$, and all the chosen vertices in $N(x)$ are distinct. Every $v_j$ is adjacent to $x$, and $v_j\in N(S_i)$ for every $i$, so $v_j$ has a neighbour in $S_i\subseteq B_i$. Hence $\{x\},B_1,\ldots,B_{s-1}$ and $\{v_1\},\ldots,\{v_t\}$ form the branch sets of a $K_{s,t}$ minor, contradicting $K_{s,t}\not\preccurlyeq G$.

\medskip

\noindent\textbf{Case 2: substantial overlap.}

\smallskip

We may therefore assume that for any collection of $s-1$ vertices $u_1,\cdots,u_{s-1} \in U$, there exist $i\neq j$ such that \begin{equation}
\label{big intersections}
|\E(u_i)\cap \E(u_j)|
>\frac{1}{s(s-2)}\min\{d(u_i),d(u_j)\}.
\end{equation}
This allows us to define an auxiliary graph $H$ with vertex set $U$ and in which distinct vertices $u,v\in U$
are adjacent if and only if 
\[
|\E (u)\cap \E (v)|>\frac{1}{s(s-2)}\min\{d(u),d(v)\}.
\]
Because any collection of $s-1$ vertices of $H$ must satisfy property
\eqref{big intersections}, we deduce that the complement graph
$\overline{H}$ contains no complete graph $K_{s-1}$ as a subgraph.

\smallskip

By Theorem \ref{thm:turan}, we have the following bound on the
cardinality on the number of edges of  $\overline{H}$: 
\[ \vert E(\overline{H }) \vert \leq \left( 1-\frac{1}{s-2}\right) \frac{\vert U\vert ^2}{2}\] and therefore, \[\vert E(H)\vert = \binom{\vert U \vert }{2}-\vert E(\overline{H})\vert \geq \frac{\vert U\vert (\vert U \vert -1)}{2}-\frac{\vert U\vert^2}{2}\left(1-\frac{1}{s-2}\right)\geq  \vert U \vert \left(\frac{\vert U \vert}{2(s-2)}-\frac{1}{2}\right) \] Rearranging, we obtain
\[
\frac{|E(H)|}{|U|}
\geq \frac{1}{2}\left(\frac{|U|}{s-2}-1\right).
\]

We now orient each edge of $H$ from its endpoint of smaller degree in $G$
towards its endpoint of larger degree, breaking ties arbitrarily.
Since $ \sum_{u\in U}d_H^+ (u)=|E(H)|$, we have \[ \frac{\sum_{u\in U}d_H^+(u)}{\vert U \vert}=\frac{|E(H)|}{|U|}
\geq \frac{1}{2}\left(\frac{|U|}{s-2}-1\right).\] Therefore, there exists $u\in U$
such that
\[
d_H^+(u)
\geq \frac{|E(H)|}{|U|}
\geq \frac{|U|}{2(s-2)}-\frac12
\geq \frac{D}{4(s-2)}-\frac12,
\]
where the last inequality follows from $|U|\geq D/2$.

\medskip

Set $q:=s(s-1)(s-2)$. For $D$ sufficiently large, the last
quantity is at least $q$. We may therefore choose distinct
outneighbours $v_1,\ldots,v_q$ of $u$ in the oriented graph $H$.
By construction, for every $i\in\{1,\ldots,q\}$, we have
$d(u)\leq d(v_i)$ and
\[
|\E(u)\cap \E(v_i)|
>\frac{1}{s(s-2)}\min\{d(u),d(v_i)\}
=\frac{d(u)}{s(s-2)}.
\]

Define $F_i:=\E(u)\cap \E(v_i)\subseteq \E(u)$ for every
$i\in\{1,\ldots,q\}$, and set
\[
r:=s-1,\qquad
m:=\frac{d(u)}{s(s-2)},\qquad
N:=|\E(u)|.
\]
Since $N\leq d(u)$, we have
\[
\frac{rN}{m}
\leq
\frac{(s-1)d(u)}{d(u)/(s(s-2))}
=s(s-1)(s-2)
=q.
\]
Thus Lemma~\ref{lem:common-intersection}, applied to
$F_1,\ldots,F_q$ as subsets of the $N$-element set $\E(u)$, gives
distinct indices $j_1,\ldots,j_{s-1}\in\{1,\ldots,q\}$ such that
\begin{align*}
|F_{j_1}\cap\cdots\cap F_{j_{s-1}}|
&\geq
|\E(u)|
\left(
\frac{d(u)}{es(s-2)|\E(u)|}
\right)^{s-1}\\
&\geq
|\E(u)|
\left(
\frac{1}{es(s-2)}
\right)^{s-1}\\
&\geq
(K-C)
\left(
\frac{1}{es(s-2)}
\right)^{s-1}\\
&\geq t,
\end{align*}
where the last inequality follows from \eqref{eq:bndKC2} (our initial
choice of $K$).

\medskip

We may therefore choose distinct vertices
$z_1,\ldots,z_t\in F_{j_1}\cap\cdots\cap F_{j_{s-1}}$.
Each $z_k$ belongs to $\E(u)$ and to $\E(v_{j_\ell})$ for every
$\ell\in\{1,\ldots,s-1\}$, and is therefore adjacent to each of
$u,v_{j_1},\ldots,v_{j_{s-1}}$.
Since $G$ is simple, none of the vertices $z_k$ coincides with any of
these $s$ vertices. Consequently, these two sets of vertices
give a copy of $K_{s,t}$ as a subgraph of $G$, a contradiction.

\medskip This concludes the proof of
Theorem~\ref{thm:maximum-degree}. \hfill $\Box$

\medskip

We note that in the proof of Theorem~\ref{thm:maximum-degree}, the
minor models of $K_{s,t}$ we have constructed all have the property
that each branch set consists of a star (a central vertex $u$ and a
subset of its neighbors). Such minor models are usually called
\emph{1-shallow minors}. We therefore obtain the following stronger
result:

\begin{theorem}\label{thm:1-shallow}
  For every integers $2\le s\le t$ and $d$, there are only finitely many connected
  locally finite graphs  with
  \begin{itemize}
    \item minimum
      degree at least $s$,
    \item maximum average degree at most $d$
    \item positive LLY curvature, and
    \item no $K_{s,t}$ as a 1-shallow minor.
    \end{itemize}
\end{theorem}

\section{Preliminaries on 5-connected plane triangulations}\label{sec:prel5ct}

In the remainder of the paper, all graphs are assumed to be simple, connected,
finite and planar, and are considered with a fixed embedding in the
plane. A graph together with such an embedding is called a
\emph{plane graph}. A plane graph is a \emph{plane triangulation} if
every face, including the unbounded face, is bounded by a triangle.
\smallskip

If $G$ is a plane graph and $v$ is a vertex of $G$, then we always list
its neighbours $v_0,\ldots,v_{d(v)-1}$ according to their cyclic order
around $v$ in the fixed embedding, with indices taken modulo $d(v)$.

\smallskip

Let $G$ be a plane graph. For a cycle $C$ of $G$, we identify $C$ with
the simple closed curve given by its embedding. By the Jordan curve
theorem, $\mathbb{R}^2\setminus C$ has exactly two connected components.
We denote the bounded and unbounded components by
$\operatorname{int}(C)$ and $\operatorname{ext}(C)$, respectively. The cycle $C$ is said to be a
\emph{separating cycle} of $G$ if both $\operatorname{int}(C)$ and
$\operatorname{ext}(C)$ contain a vertex of $G$.

\medskip

For an integer $k\geq 1$, a graph $G$ is said to be
\emph{$k$-connected} if $|V(G)|\geq k+1$ and $G\setminus X$ is
connected for every $X\subseteq V(G)$ with $|X|<k$.
Thus, if a plane graph is $5$-connected, it has no separating cycle of
length $3$ or $4$. 
We will repeatedly use the following observation on the local structure of 5-connected planar triangulations, which can be routinely verified. %See Figure \ref{fig:degree-sixteen-neighborhood} for an illustration. 
%\textcolor{red}{Enlever cette figure et mettre celle de $G_{a,b}$ à la place.}

\begin{observation}[Local configuration around a vertex]
\label{obs:local-structure}
Let $G$ be a $5$-connected plane triangulation, and let
$v\in V(G)$ be a vertex of degree $a$. Write its neighbours as
$v_0,v_1,\ldots,v_{a-1}$, with the cyclic ordering and indexing
convention fixed above (all the indices being considered modulo $a$). Then the following hold.

\begin{enumerate}
    \item $a\geq 5$.

    \item Two vertices $v_i,v_j$ are  adjacent in $G$ if and only if they are consecutive (that is, $|i-j|=1$).

    \item For every $i$, $v_i$ and $v_{i+1}$  have
    exactly two common neighbours. One of them is $v$, and the other
    lies outside $N[v]$.

    \item For every $i$, $v_i$ and $v_{i+1}$ have exactly two common neighbors: $v$ and $v_{i+1}$.

    \item For every pair $i,j$ with $|i-j|> 2$, $v_i$ and $v_j$ have a unique common neighbor: $v$.
\end{enumerate}
\end{observation}

%\begin{figure}[htb]
%    \centering
%    \includegraphics[width=0.3\textwidth]{neighborhood.pdf}
%    \caption{The subgraph induced by $N[v]$ around a vertex $v$ of
%    degree $16$. The neighbours $v_0,v_1,\ldots,v_{15}$, listed in
%    cyclic order around $v$, induce the cycle
%    $v_0v_1\cdots v_{15}v_0$.}
%    \label{fig:degree-sixteen-neighborhood}
%\end{figure}

\begin{figure}[htb]
    \centering
    \includegraphics[width=0.3\textwidth]{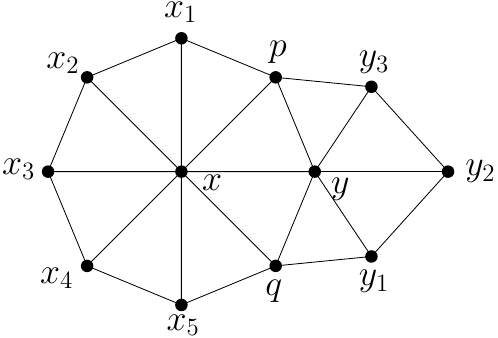}
    \caption{The induced configuration $G_{8,6}$ around an edge
    $xy$ with $d(x)=8$ and $d(y)=6$.
    The vertices $p,q$ are the common neighbours of $x$ and $y$.}
    \label{fig:local-G86}
\end{figure}

We obtain the following as a direct consequence. See Figure \ref{fig:local-G86} for an illustration.

 \begin{lemma}[Local configuration around an edge]
\label{lem:local-edge-configuration}
Let $G$ be 5-connected triangulation, and let
$xy\in E(G)$, with $d(x)=a$ and $d(y)=b$.
Let $p,q$ be the third vertices of the faces incident with $xy$.
Then $N_G(x)\cap N_G(y)=\{p,q\}$, and we may label the neighbours
of $x$ and $y$ in cyclic order as
$y,p,x_1,\ldots,x_{a-3},q$ and
$x,q,y_1,\ldots,y_{b-3},p$, respectively. Thus, the induced graph $G_{a,b}:=G[N[x]\cup N[y]]$ has vertex set
$$
V(G_{a,b})
=\{x,y,p,q,x_1,\ldots,x_{a-3},y_1,\ldots,y_{b-3}\}.
$$
Each neighbourhood induces a cycle, and the edges of $G_{a,b}$
are precisely the edges of these two cycles and those incident
with $x$ or $y$.
For $u,w\in V(G_{a,b})$, the distances $d_G(u,w)$ and
$d_{G_{a,b}}(u,w)$ agree whenever $d_{G_{a,b}}(u,w)\leq 2$.
If $d_{G_{a,b}}(u,w)=3$, then $d_G(u,w)\in\{2,3\}$.
\end{lemma}

\begin{proof}
The common-neighbour assertion and the induced neighbourhood
cycles follow from Observation~\ref{obs:local-structure}.
An edge $x_i y_j$ would make $y_j$ a common neighbour of
the nonconsecutive neighbours $y,x_i$ of $x$, lying outside
$N[x]$, contrary to parts~(5)--(6) of that lemma.
Thus no additional edges occur.

Since $G_{a,b}$ is induced, adjacency is the same in both
graphs. Together with $d_G(u,w)\leq d_{G_{a,b}}(u,w)$,
this proves the distance assertions.
\end{proof}

We conclude this section with a bound on the curvature of an edge in a 5-connected triangulation as a function of the degrees of its endpoints.

\begin{lemma}[A curvature estimate]
\label{lem:curvature-estimate}
Let $G$ be a 5-connected plane triangulation, and let $xy\in E(G)$
with $d(x)=a$ and $d(y)=b$. Then
$$
\kappa_{\mathrm{LLY}}(x,y)\leq \frac{4}{a}+\frac{3}{b}-1.
$$
\end{lemma}

\begin{proof}
Using the notation of Lemma~\ref{lem:local-edge-configuration},
define $f$ on $N(x)\cup N(y)$ by
$$
f(x)=f(p)=f(q)=0,\qquad f(y)=1,\qquad
f(x_i)=-1,\qquad f(y_j)=1.
$$
By that lemma, $f$ extends to a $1$-Lipschitz function on $V(G)$.
Since $f(y)-f(x)=1$, Theorem \ref{thm:limit-free} (the limit-free characterization of LLY curvature) gives
$$
\kappa_{\mathrm{LLY}}(x,y)
\leq \Delta f(x)-\Delta f(y)
=\left(\frac{4}{a}-1\right)+\frac{3}{b}.
$$
\end{proof}
\section{The 5-connected plane triangulations with positive LLY curvature}\label{sec:5ct}

In this section, we prove Theorem \ref{thm:5ct}.
Consider the five 5-connected triangulations depicted in Figures \ref{fig:G12}, \ref{fig:G14}, \ref{fig:G15}, \ref{fig:G16}, and \ref{fig:G17}. In each figure, the red vertices have degree $6$, whereas the black vertices have degree $5$. Green edges have LLY curvature equal to $\tfrac25$, blue edges have LLY curvature equal to $\tfrac2{15}$, and purple edges have LLY curvature equal to $\tfrac16$.

\begin{figure}[htb]
    \centering
    \begin{minipage}[t]{0.48\textwidth}
        \centering
        \includegraphics[width=0.65\linewidth]{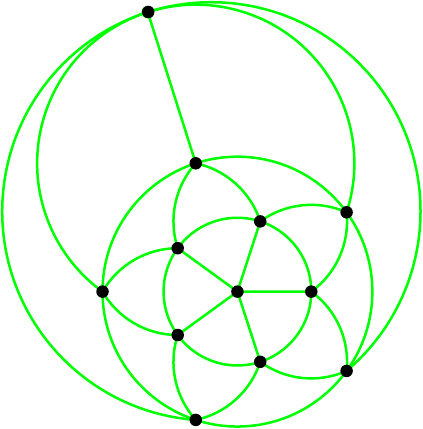}
        \caption{The icosahedron graph $G_{12}$.}
        \label{fig:G12}
    \end{minipage}\hfill
    \begin{minipage}[t]{0.48\textwidth}
        \centering
        \includegraphics[width=0.65\linewidth]{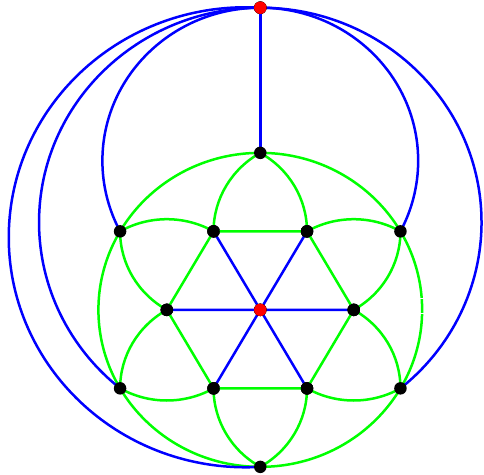}
        \caption{The graph $G_{14}$.}
        \label{fig:G14}
    \end{minipage}
\end{figure}

\begin{figure}[htb]
    \centering
    \begin{minipage}[t]{0.48\textwidth}
        \centering
        \includegraphics[width=0.65\linewidth]{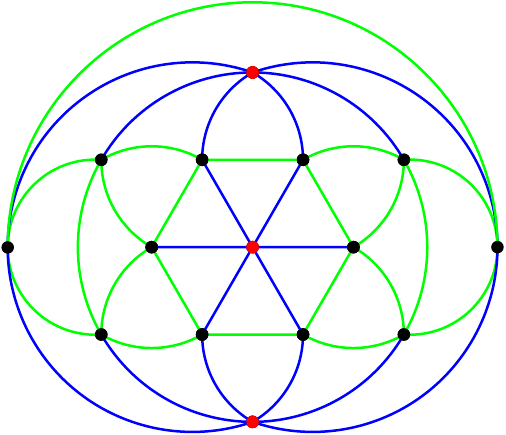}
        \caption{The graph $G_{15}$.}
        \label{fig:G15}
    \end{minipage}\hfill
    \begin{minipage}[t]{0.48\textwidth}
        \centering
        \includegraphics[width=0.65\linewidth]{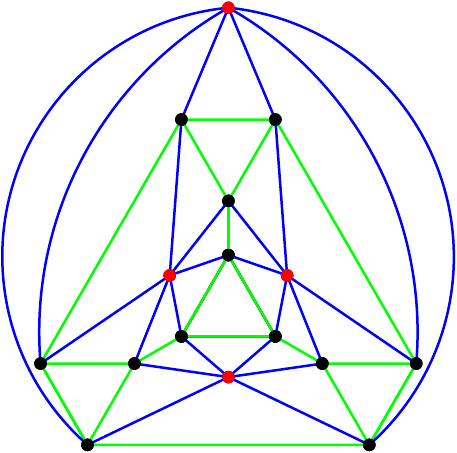}
        \caption{The graph $G_{16}$.}
        \label{fig:G16}
    \end{minipage}
\end{figure}

\begin{figure}[htb]
    \centering
    \includegraphics[width=0.3\textwidth]{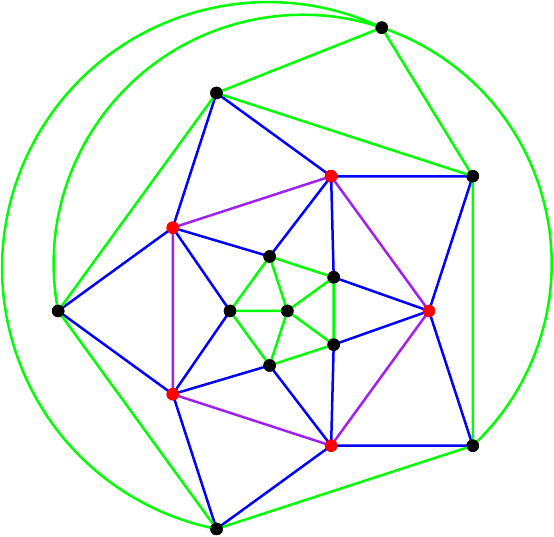}
    \caption{The graph $G_{17}$.}
    \label{fig:G17}
\end{figure}

\smallskip

The five triangulationss are all 5-connected, have positive LLY curvature, and every vertex has degree  $5$ or $6$. We will show that these are in fact the only examples.
The proof is divided into four steps, corresponding to the following properties.

\medskip

\noindent $\boldsymbol{\mathcal P_1}$:   \textbf{The maximum degree is at most $6$.}
    This is done by excluding all larger degrees. For sufficiently large degrees, we construct suitable $1$-Lipschitz functions giving non-positive upper bounds for the curvature. The remaining cases are excluded by combining curvature constraints with restrictions imposed by planarity and 5-connectivity.

    \medskip

\noindent $\boldsymbol{\mathcal P_2}$:
    \textbf{The local configurations around vertices of degree $5$ and $6$ are strongly restricted.}
    Let $v$ be a vertex of degree $a\in\{5,6\}$, and write its neighbours in cyclic order as
    \[
        v_0,\ldots,v_{a-1},
    \]
    with indices understood modulo $a$. We prove the following.
    \begin{enumerate}
        \item If $d(v)=6$, then exactly one of the following occurs:
        \begin{enumerate}
            \item all six neighbours of $v$ have degree $5$;
            \item exactly two neighbours of $v$ have degree $6$, and they are antipodal in the cyclic order, that is, they are of the form $v_i$ and $v_{i+3}$ for some $i\in\mathbb Z/6\mathbb Z$.
        \end{enumerate}

        \item If $d(v)=5$, then $v$ has at most two neighbours of degree $6$. Consequently, according to the number and relative positions of its degree $6$ neighbours, only 3 many local configurations are possible, up to cyclic reordering.
    \end{enumerate}

    \medskip

    \noindent $\boldsymbol{\mathcal P_3}$:
    \textbf{The graph has at most $17$ vertices.}
    Using $\Delta(G)\leq 6$, Euler's formula, and a discharging argument, we show that $
        |V(G)|\leq 17$.

        \medskip

\noindent $\boldsymbol{\mathcal P_4}$:
    \textbf{The preceding properties determine the graph completely.}
    Combining $\mathcal P_1$--$\mathcal P_3$ with a final discharging argument, we show that $G$ is isomorphic to one of
    \[
        G_{12},\quad G_{14},\quad G_{15},\quad G_{16},\quad G_{17}.
    \]

\medskip

We prove the four properties in separate subsections. Together, they complete the proof of the Theorem \ref{thm:5ct}.

\medskip

We note that we could have also stopped right after proving Property $\mathcal{P}_3$ (or any reasonable bound on the order of $G$; 24 would be enough for instance, and proving it does not require Property $\mathcal{P}_2$ in its entirety). It would then suffice to check all 5-connected plane triangulations of order at most 24 and verify that among them, only $G_{12}$, $G_{14}$, $G_{15}$, $G_{16}$, and $G_{17}$ have positive LLY curvature. But this would not really explain why these are the only one, so we found it more satisfactory to actually \emph{prove} that these are the only one. 

\subsection{Property $\mathcal P_1$: bounding the maximum degree}

\begin{proposition}\label{prop:degree}
    If $G$ is a 5-connected plane triangulation with positive LLY curvature, then $\Delta(G) \leq 6$.
\end{proposition}
\begin{proof}
Let $G$ is a 5-connected plane triangulation with positive LLY curvature.    Set $\Delta(G)=a$, and let $v$ be a vertex of degree $a$. Suppose for contradiction that $a \geq 7$.

\medskip
\noindent
\textbf{Case 1: $a\geq 10$.}

\noindent Let $y$ be a neighbour of $v$, and write $d(y)=b\geq 5$.
By Lemma~\ref{lem:curvature-estimate},
$$
\kappa_{\mathrm{LLY}}(v,y)\leq \frac{4}{a}+\frac{3}{b}-1
\leq \frac{4}{10}+\frac{3}{5}-1=0,
$$
contradicting positive curvature. Therefore, $\Delta(G)\leq 9$.

\medskip
\noindent
\textbf{Case 2: $a\in\{8,9\}$.}

\noindent Suppose first that $v$ has a neighbour $y$ of degree $b\geq 6$.
By Lemma~\ref{lem:curvature-estimate},
$$
\kappa_{\mathrm{LLY}}(v,y)\leq \frac{4}{a}+\frac{3}{b}-1
\leq \frac{4}{8}+\frac{3}{6}-1=0,
$$
a contradiction. Thus every neighbour of $v$ has degree $5$.

\smallskip

Write the neighbours of $v$ in cyclic order as
$v_0,\ldots,v_{a-1}$, with indices taken modulo $a$.
For each $i$, let $w_i$ be the third vertex of the face incident
with $v_iv_{i+1}$ on the side not containing $v$. See Figures \ref{fig:degree-nine-all-five} and \ref{fig:degree-eight-all-five} for an illustration. Then
$$
N(v_i)=\{v,v_{i-1},v_{i+1},w_{i-1},w_i\}.
$$
By Observation~\ref{obs:local-structure}, the vertices $w_i$
lie outside $N[v]$ and are pairwise distinct.

\medskip

We next show that some pairs of vertices $w_i$ and $v_j$ which are at distance at least 3 in the ball of radius 2 centered in $v$ must be at distance 2 in $G$, and thus the cycle $w_0,w_1,\ldots,w_{a-1}$ must have some chords. The corresponding sets of vertices $v_i$ and $w_j$ are identified in blue in Figures \ref{fig:degree-nine-all-five} and \ref{fig:degree-eight-all-five}. By applying this claim several times we will obtain crossing chords, contradicting the planarity of $G$.
\begin{figure}[htb]
    \centering
    \begin{minipage}[t]{0.48\textwidth}
        \centering
        \includegraphics[width=0.65\linewidth]{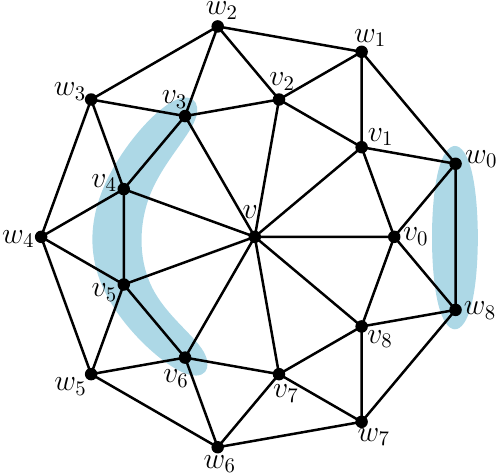}
        \caption{The local configuration around a vertex of degree $9$.}
        \label{fig:degree-nine-all-five}
    \end{minipage}\hfill
    \begin{minipage}[t]{0.48\textwidth}
        \centering
        \includegraphics[width=0.65\linewidth]{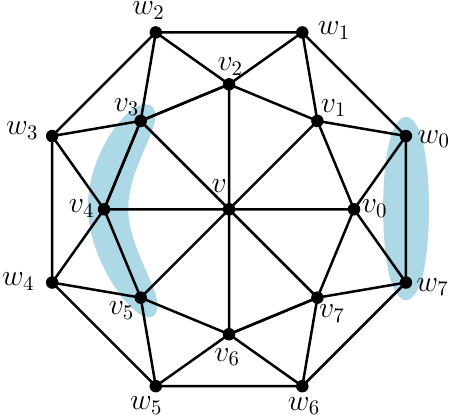}
        \caption{The local configuration around a vertex of degree $8$.}
        \label{fig:degree-eight-all-five}
    \end{minipage}
\end{figure}
\medskip
\begin{claim}\label{cl:72}
At least one of the distances $d_G(w_{a-1},v_j)$ or
$d_G(w_0,v_j)$, for $3\leq j\leq a-3$, is equal to $2$.    
\end{claim}

\noindent
\emph{Proof of the claim.}
Suppose, for a contradiction, that none of these distances equals $2$.
Define $f$ on $N[v]\cup N[v_0]$ by
$$
\begin{aligned}
f(v)=f(v_2)=f(v_{a-2})&=0,
& f(v_0)=f(v_1)=f(v_{a-1})&=1,\\
f(v_j)&=-1\quad(3\leq j\leq a-3),
& f(w_{a-1})=f(w_0)&=2.
\end{aligned}
$$
The only pairs whose values differ by $3$ are those listed in
the claim. Their distances are $3$ by our assumption and
Lemma~\ref{lem:local-edge-configuration}. Pairs whose values
differ by $2$ are nonadjacent in the induced local configuration.
Hence $f$ extends to a $1$-Lipschitz function on $V(G)$, and
$$
\kappa_{\mathrm{LLY}}(v,v_0)\leq \Delta f(v)-\Delta f(v_0)
=\frac{8-a}{a}-\frac{1}{5}<0,
$$
contradicting positive curvature. This proves the claim. \hfill $\blacksquare$

\medskip

It was convenient to fix $w_0,w_{a-1}$ when stating Claim \ref{cl:72}, but by symmetry it can be applied in the same way to any pair of consecutive vertices $w_i,w_{i+1}$, by appropriately shifting the indices of the corresponding vertices $v_j$.

\medskip

Since each $v_i$ has degree $5$, triangulation gives
$w_{i-1}w_i\in E(G)$ for every $i$. Thus the vertices $w_i$
form a cycle $C=w_0w_1\cdots w_{a-1}w_0$ surrounding $v$.
The prescribed triangular faces fill the side of $C$ containing
$v$, so all possible chords of $C$ lie on the other side.
By 5-connectivity, each vertex of $C$ has at least one neighbor in $\mathrm{ext}(C)$ and no chord joins vertices at cyclic distance $2$ or $3$ in $C$.

\smallskip
\noindent
\textbf{Suppose first that $a=9$.}
By Claim~\ref{cl:72}, at least one of
$w_0w_4$, $w_0w_5$, $w_8w_3$, $w_8w_4$ belongs to $E(G)$;
the other possible chords have cyclic distance $2$ or $3$.
These four possibilities are equivalent up to cyclic relabelling
and reversal, so we may assume that $w_0w_4\in E(G)$.
Applying Claim~\ref{cl:72} to $w_1,w_2$ then forces at least one of
$w_2w_6$, $w_2w_7$, $w_1w_5$, $w_1w_6$ to belong to $E(G)$.
Each of these edges crosses $w_0w_4$, since their endpoints
alternate in the cyclic order on $C$. This contradicts planarity.
Hence $a\neq 9$.

\smallskip
\noindent
\textbf{Suppose now that $a=8$.}
Applying Claim~\ref{cl:72} again gives $w_0w_4\in E(G)$ or
$w_7w_3\in E(G)$. After cyclic relabelling, we may assume
that $w_0w_4\in E(G)$.
Applying Claim~\ref{cl:72} again to $w_1,w_2$ forces $w_2w_6\in E(G)$
or $w_1w_5\in E(G)$. Both edges cross $w_0w_4$, again
contradicting planarity. Therefore, $a\leq 7$.

\medskip
\noindent
\textbf{Case 3: $a=7$.}

\noindent It remains to exclude vertices of degree $7$.
First, no two such vertices are adjacent. Indeed, if
$xy\in E(G)$ and $d(x)=d(y)=7$, then
Lemma~\ref{lem:curvature-estimate} gives
$$
\kappa_{\mathrm{LLY}}(x,y)\leq \frac{4}{7}+\frac{3}{7}-1=0,
$$
contradicting positive curvature. Thus every neighbour of $v$
has degree $5$ or $6$.

We use the same notation $v_i,w_i$ around $v$, now with indices
taken modulo $7$. If $d(v_i)=6$, let $t_i$ be its remaining
neighbour, so that its neighbours, in cyclic order, are
$v,v_{i-1},w_{i-1},t_i,w_i,v_{i+1}$.
If $d(v_i)=5$, the vertex $t_i$ is omitted. We use this convention
in all the sets below.

\medskip

\begin{figure}[htb]
    \centering
    \includegraphics[width=0.3\textwidth]{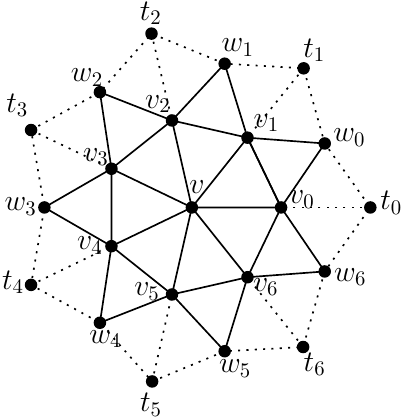}
    \caption{The local configuration when $d(v)=7$. The dashed edges
    incident with $t_i$ are present only when $d(v_i)=6$. If $d(v_i)=5$, then $t_i$ does not exist and $w_i$ and $w_{i+1}$ are adjacent.}
    \label{fig:loc7}
\end{figure}

As before, the vertices $w_i$ lie outside $N[v]$ and are pairwise distinct.
Moreover, Observation~\ref{obs:local-structure} shows that
each $t_i$ lies outside $N[v]$ and is adjacent to no neighbour
of $v$ other than $v_i$. Hence all the $w_i$ and the existing
$t_i$ are pairwise distinct. Figure~\ref{fig:loc7} contains all
vertices at distance at most $2$ from $v$; the dashed edges incident
with $t_i$ are present only when $d(v_i)=6$.

\begin{claim}\label{cl:73}
    The following hold:
\begin{enumerate}
    \item At least one of $d_G(w_6,v_3)$, $d_G(w_6,v_4)$,
    $d_G(w_0,v_3)$ and $d_G(w_0,v_4)$ is equal to $2$.
    \item If $d(v_0)=6$, then $d_G(t_0,v_j)=2$ for some
    $j\in\{2,3,4,5\}$.
\end{enumerate}

\end{claim} 
\smallskip
\noindent
\emph{Proof of the claim.}
Write $b=d(v_0)\in\{5,6\}$.
For the first assertion, suppose that none of the four distances
equals $2$. Use the function $f$ from the proof of Claim~\ref{cl:73} with
$a=7$, and set $f(t_0)=1$ if $b=6$.
The same argument shows that $f$ extends to a $1$-Lipschitz
function on $V(G)$. Moreover,
$$
\kappa_{\mathrm{LLY}}(v,v_0)\leq \Delta f(v)-\Delta f(v_0)
=\frac{1}{7}-\frac{1}{b}<0,
$$
a contradiction. This proves the first assertion for both
$b=5$ and $b=6$.

For the second assertion, suppose that $b=6$ and that
$d_G(t_0,v_j)\neq 2$ for every $j\in\{2,3,4,5\}$.
Define $g$ on $N[v]\cup N[v_0]$ by
$$
\begin{aligned}
g(v)=g(v_1)=g(v_6)&=0,
&g(v_0)=g(w_6)=g(w_0)&=1,\\
g(v_j)&=-1\quad(j\in\{2,3,4,5\}),
&g(t_0)&=2.
\end{aligned}
$$
The only pairs whose values differ by $3$ are $(t_0,v_j)$
for $j\in\{2,3,4,5\}$. Their distances are $3$ by assumption,
so the same argument shows that $g$ extends to a $1$-Lipschitz
function on $V(G)$. Consequently,
$$
\kappa_{\mathrm{LLY}}(v,v_0)\leq \Delta g(v)-\Delta g(v_0)
=-\frac{3}{7}+\frac{1}{3}=-\frac{2}{21}<0,
$$
again contradicting positive curvature. This proves the claim.\hfill $\blacksquare$

\medskip

As before, we have chosen $w_0,w_6$ to state the claim but it can be applied to any pair of consecutive vertices $w_i,w_{i+1}$ by symmetry.

\medskip

 The vertices $w_i$
and the existing $t_i$ form a cycle $C$, in which $t_i$ lies
between $w_{i-1}$ and $w_i$ whenever it exists.
As in Case~2, all chords lie on the side not containing $v$.
The same argument excludes chords joining vertices at cyclic
distance $2$ or $3$: an empty triangle or quadrilateral would
give an intermediate vertex degree at most $4$.
The distance conditions in Claim~\ref{cl:73} must therefore be realized
by chords of $C$.

\smallskip

We shall also use the fact that, whenever $t_i$ exists,
\begin{equation}
\label{eq:forbidden-near-t-edges}
t_iw_{i+1},\ t_iw_{i-2}\notin E(G),
\end{equation}
which is a direct consequence of Observation~\ref{obs:local-structure}.

\medskip

We now consider two cases.

\smallskip
\noindent
\textbf{Subcase 3.1: $v$ has at most one neighbour of degree $6$.}

\noindent Suppose first that all neighbours of $v$ have degree $5$. This case is illustrated in Figure \ref{fig:all-degree-five-neighbours}.
Then $C$ has length $7$, so every chord joins vertices at cyclic
distance $2$ or $3$. Thus $C$ has no chords, contradicting
Claim~\ref{cl:73}(1).

\smallskip

Suppose now that $v$ has exactly one neighbour of degree $6$,
which we may assume is $v_0$. This case is illustrated in Figure \ref{fig:one-degree-six-neighbour}. Then $C$ has cyclic order
$t_0,w_0,w_1,\allowbreak\ldots,w_6$.
Applying Claim~\ref{cl:73}(2) forces $t_0$ to have a common
neighbour with one of $v_2,v_3,v_4,v_5$.
The only possible chord that does not join vertices at cyclic
distance $2$ or $3$ is $t_0w_3$. Consequently,
$t_0w_3\in E(G)$.
Applying Claim~\ref{cl:73}(1) to $w_0,w_1$ then forces $w_0w_4\in E(G)$
or $w_1w_5\in E(G)$. Both edges cross $t_0w_3$, a contradiction.
Thus this subcase is impossible.

\begin{figure}[!htbp]
    \centering
    \begin{minipage}[t]{0.48\textwidth}
        \centering
        \includegraphics[width=0.65\linewidth]
            {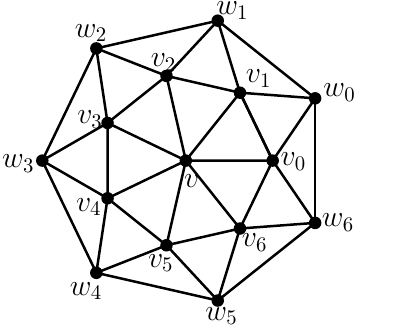}
        \caption{All neighbours of $v$ have degree $5$.}
        \label{fig:all-degree-five-neighbours}
    \end{minipage}\hfill
    \begin{minipage}[t]{0.48\textwidth}
        \centering
        \includegraphics[width=0.65\linewidth]{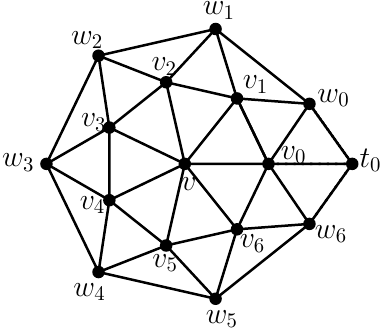}
        \caption{Exactly one neighbour of $v$ has degree $6$.}
        \label{fig:one-degree-six-neighbour}
    \end{minipage}
\end{figure}

\smallskip
\noindent
\textbf{Subcase 3.2: $v$ has at least two neighbours of degree $6$.}

\noindent The cyclic distance between two such neighbours is $1$, $2$ or $3$.
We exclude these three possibilities in turn.

\smallskip
\noindent
\emph{Two degree $6$ neighbours have cyclic distance $2$.}

After cyclic relabelling, suppose that $d(v_0)=d(v_2)=6$. This is illustrated in Figures \ref{fig:degree-six-0-2-case} and \ref{fig:degree-six-0-2-v1-five}.
Applying Claim~\ref{cl:73}(2) to $w_0,w_6$ and $w_1,w_2$, and using
\eqref{eq:forbidden-near-t-edges}, forces edges
$t_0x_0,t_2x_2\in E(G)$ with
$$
\begin{aligned}
x_0&\in\{t_2,w_2,t_3,w_3,t_4,w_4,t_5\},\\
x_2&\in\{t_4,w_4,t_5,w_5,t_6,w_6,t_0\}.
\end{aligned}
$$

\begin{figure}[!htbp]
    \centering
    \begin{minipage}[t]{0.48\textwidth}
        \centering
        \includegraphics[width=0.65\linewidth]{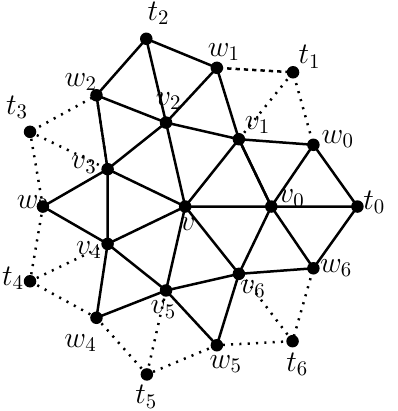}
        \caption{The configuration when $d(v_0)=d(v_2)=6$.}
        \label{fig:degree-six-0-2-case}
    \end{minipage}\hfill
    \begin{minipage}[t]{0.48\textwidth}
        \centering
        \includegraphics[width=0.65\linewidth]{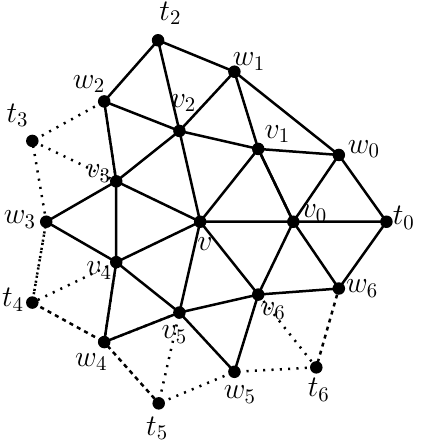}
        \caption{The same configuration when $d(v_1)=5$.}
        \label{fig:degree-six-0-2-v1-five}
    \end{minipage}
\end{figure}

We first exclude $t_0t_2\in E(G)$.
If $d(v_1)=5$, this would be a chord joining vertices at cyclic
distance $3$ in $C$.
If $d(v_1)=6$, applying Claim~\ref{cl:73}(2) to $w_0,w_1$ and using
\eqref{eq:forbidden-near-t-edges} forces an edge $t_1x_1$ with
$x_1\in\{t_3,w_3,t_4,w_4,t_5,w_5,t_6\}$.
Every such edge crosses $t_0t_2$, which is impossible.
Therefore, $x_0\neq t_2$ and $x_2\neq t_0$.

\smallskip

By planarity, the edges $t_0x_0$ and $t_2x_2$ can be drawn
simultaneously only if
$$
x_0,x_2\in\{t_4,w_4,t_5\}.
$$
Indeed, on the arc from $t_2$ to $t_0$ containing $w_4$,
$x_2$ must occur before $x_0$, unless they coincide.

Applying Claim~\ref{cl:73}(1) to $w_5,w_6$ now forces an edge from $w_5$ or
$w_6$ to a vertex in $\{w_1,t_2,w_2,t_3,w_3\}$.
Each such edge crosses $t_0x_0$, since its endpoints lie on
opposite arcs of $C$ determined by $t_0$ and $x_0$.
This contradicts planarity. Hence, two degree $6$ neighbours
of $v$ cannot have cyclic distance $2$.

\medskip
\noindent
\emph{Two degree $6$ neighbours are consecutive.}

\noindent After cyclic relabelling, suppose that $d(v_0)=d(v_1)=6$. This is illustrated in Figure \ref{fig:consecutive-degree-six-neighbours}.
By the previous case, $v_2,v_3,v_5,v_6$ have degree $5$,
while $v_4$ may have degree $5$ or $6$.
Applying Claim~\ref{cl:73}(2) to $w_0,w_0$ and $w_0,w_1$, and excluding the
edges in \eqref{eq:forbidden-near-t-edges} and the short chords,
forces edges $t_0x_0,t_1x_1\in E(G)$ with
$$
x_0\in\{w_2,w_3,t_4\},\qquad
x_1\in\{t_4,w_4,w_5\}.
$$
If $d(v_4)=5$, then $t_4$ is absent, and every choice of these
two edges produces a crossing. Hence $d(v_4)=6$.
The only noncrossing possibility is $x_0=x_1=t_4$, so both
$t_0t_4$ and $t_1t_4$ belong to $G$.

\smallskip

\begin{figure}[!htbp]
    \centering
    \includegraphics[width=0.3\textwidth]{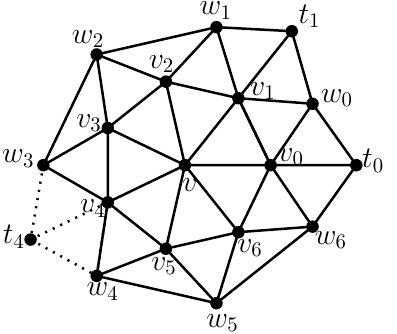}
    \caption{The local configuration when $d(v_0)=d(v_1)=6$.}
    \label{fig:consecutive-degree-six-neighbours}
\end{figure}

Applying Claim~\ref{cl:73}(1) to $w_2,w_3$ forces an edge between
$\{w_2,w_3\}$ and $\{w_5,w_6,t_0,w_0\}$.
Each such edge crosses $t_1t_4$, a contradiction.
Therefore, $v$ cannot have two consecutive neighbours of degree $6$.

\medskip
\noindent
\emph{Two degree $6$ neighbours have cyclic distance $3$.}

\noindent After cyclic relabelling, suppose that $d(v_0)=d(v_3)=6$. This is illustrated in Figure \ref{fig:degree-six-0-3}.
By the previous two cases, all the remaining neighbours of $v$
have degree $5$.
Applying Claim~\ref{cl:73}(2) to $w_0,w_6$, and excluding the short chords,
forces $t_0z_0\in E(G)$ for some $z_0\in\{t_3,w_3\}$.

\smallskip

\begin{figure}[!htbp]
    \centering
    \includegraphics[width=0.3\textwidth]{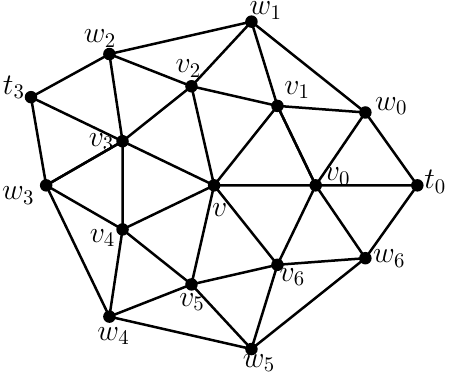}
    \caption{The local configuration when $d(v_0)=d(v_3)=6$.}
    \label{fig:degree-six-0-3}
\end{figure}

Applying Claim~\ref{cl:73}(1) to $w_1,w_2$ then forces an edge between
$\{w_1,w_2\}$ and $\{w_4,w_5,w_6\}$.
Each such edge crosses $t_0z_0$, again contradicting planarity.

\medskip

We have therefore excluded every possible arrangement of the
degree $6$ neighbours of $v$. Indeed, $v$ cannot have at most
one such neighbour, and any two would have cyclic distance
$1$, $2$ or $3$, all of which have been excluded.
This contradiction shows that $G$ has no vertex of degree $7$.
Therefore, $\Delta(G)\leq 6$, which concludes the proof. \end{proof}

\subsection{Property $\mathcal{P}_2$: the possible local configurations}
By Proposition \ref{prop:degree} we know that in the graphs under consideration, the only possible vertex degrees are $5$ and $6$. We shall fix some notation for the rest of the section: 

Let $v$ be a vertex of degree $a\in \{5,6\}$. Write the neighbours of $v$ in cyclic order as $v_0,\ldots,v_{a-1}$ (where indices are considered modulo $a$). Each $v_i$ has degree either $5$ or $6$. For each $i$, let
$w_i$ is the third vertex of the face incident with $v_iv_{i+1}$
on the side not containing $v$, and, whenever $d(v_i)=6$,
\[
N(v_i)=\{v,v_{i-1},w_{i-1},t_i,w_i,v_{i+1}\}
\]
in cyclic order. If $d(v_i)=5$, the vertex $t_i$ is omitted.

By Observation~\ref{obs:local-structure}, the vertices $w_i$
and the existing $t_i$ lie outside $N[v]$ and are pairwise
distinct. 
They form a cycle $C$, in which $t_i$ lies between
$w_{i-1}$ and $w_i$ whenever it exists. 
We use this convention in all the sets below.

\medskip

We now determine the possible
arrangements of these degrees around a vertex. The following
restrictions will be used in the next subsection to bound
$|V(G)|$.
We first consider a vertex of degree $6$.

\begin{proposition}
\label{prop:local-configurations6}
Let $G$ be a 5-connected plane triangulation with positive
LLY curvature. Let $v$ be a vertex of degree $6$. 
Then, either all of the neighbors of $v$ have degree $5$ or $v$ has exactly two neighbors of degree $6$ and they are antipodal in the cyclic order.
\end{proposition}

\begin{proof}
Let $v\in V(G)$ be a vertex of degree $6$. We use the notation introduced above. See Figure \ref{fig:degree-six} for an illustration.
Let $C$ be the cycle formed by the vertices $w_i$ and $t_j$. The faces incident with
vertices of $N[v]$ triangulate the disk $D$ bounded by $C$ and
containing $v$. Every vertex of $C$ is incident with at most four
edges of this disk triangulation (and thus at least one edge outside of the disk), and every chord of $C$ lies
outside $D$.

\begin{figure}[!htbp]
    \centering
    \includegraphics[width=0.3\textwidth]{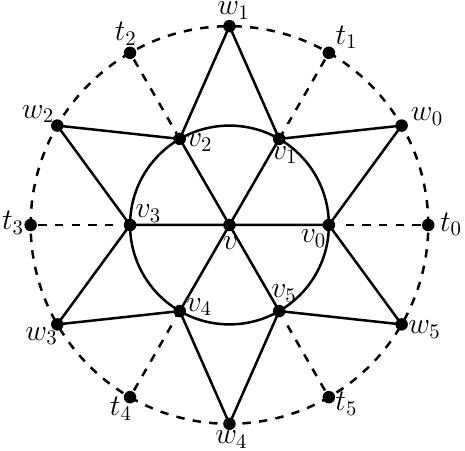}
    \caption{The local configuration around $v$ when
    $d(v)=6$. For each $i$, the vertex $t_i$
    is present precisely when $d(v_i)=6$; otherwise the
    corresponding boundary segment is the edge $w_{i-1}w_i$.}
    \label{fig:degree-six}
\end{figure}

By $5$-connectivity, no chord joins vertices at cyclic
distance $2$ or $3$ in $C$. 

\begin{claim}
\label{cl:degree-six-distance}
If $d(v_i)=6$, then $d_G(t_i,v_j)=2$ for some
$j\in\{i+2,i+3,i+4\}$.
\end{claim}

\noindent \textit{Proof of the claim.}
Suppose otherwise. Define $f$ on $N[v]\cup N[v_i]$ by
\[
\begin{aligned}
f(v)=f(v_{i-1})=f(v_{i+1})&=0,
& f(t_i)&=2,\\
f(v_i)=f(w_{i-1})=f(w_i)&=1,
& f(v_{i+2})=f(v_{i+3})=f(v_{i+4})&=-1.
\end{aligned}
\]
The induced local configuration around $vv_i$ shows that pairs
whose values differ by $2$ are nonadjacent. The only pairs whose
values differ by $3$ are $(t_i,v_j)$ for
$j\in\{i+2,i+3,i+4\}$, and their distances are at least $3$
by assumption. Thus $f$ extends to a $1$-Lipschitz function on
$V(G)$. Since $f(v_i)-f(v)=1$ and
\[
\Delta f(v)=-\frac13,
\qquad
\Delta f(v_i)=\frac46-1=-\frac13,
\]
we obtain $\kappa_{\mathrm{LLY}}(v,v_i)\leq\Delta f(v)-\Delta f(v_i)=0$,
contrary to positive curvature.
\hfill $\blacksquare$

\medskip

We also have $t_iw_{i+1},t_iw_{i-2}\notin E(G)$ whenever $t_i$
exists. Indeed, Observation~\ref{obs:local-structure} implies  $N(t_i)\cap N(v_{i+1})=\{v_i,w_i\}$;
the other side is symmetric. Claim~\ref{cl:degree-six-distance}
and the neighbourhoods of $v_{i+2},v_{i+3},v_{i+4}$ therefore
force a chord
\begin{equation}
\label{eq:degree-six-forced-chords}
t_ix_i\in E(G),
\qquad
x_i\in\{t_{i+2},w_{i+2},t_{i+3},w_{i+3},t_{i+4}\}.
\end{equation}

We first exclude $t_it_{i+2}\in E(G)$ whenever both vertices
exist. If $d(v_{i+1})=5$, this chord cuts off the three-edge path
$t_iw_iw_{i+1}t_{i+2}$, which is impossible. Otherwise,
$t_{i+1}$ lies on the $C$-arc from $t_i$ to $t_{i+2}$
containing $w_i$, whereas all its possible endpoints in
\eqref{eq:degree-six-forced-chords} lie on the complementary
open arc from $t_{i+3}$ to $t_{i+5}$, so the required chord would therefore cross $t_it_{i+2}$.
This argument applies to every $i$ and also excludes $t_it_{i+4}$.
Consequently, \eqref{eq:degree-six-forced-chords} reduces to
\begin{equation}
\label{eq:degree-six-reduced-chords}
t_ix_i\in E(G),
\qquad
x_i\in\{w_{i+2},t_{i+3},w_{i+3}\}.
\end{equation}

\medskip
\noindent\emph{At most one neighbour of degree $6$.}

\noindent If all neighbours of $v$ have degree $5$, the first alternative
holds. If exactly one has degree $6$, then $C$ has length $7$,
so every chord joins vertices at cyclic distance $2$ or $3$.
Thus $C$ has no chords, contradicting
\eqref{eq:degree-six-reduced-chords}.

\medskip
\noindent\emph{At least two neighbours of degree $6$.}

\noindent The cyclic distance between two such neighbours is $1$, $2$
or $3$. We exclude the first two possibilities.

\smallskip
\noindent\emph{Two degree $6$ neighbours have cyclic distance $2$.}
After cyclic relabelling, suppose that $d(v_0)=d(v_2)=6$. This is illustrated in Figure \ref{fig:degree-six-0-2}.
By \eqref{eq:degree-six-reduced-chords}, there are chords
$t_0x_0,t_2x_2\in E(G)$ with
\[
x_0\in\{w_2,t_3,w_3\},
\qquad
x_2\in\{w_4,t_5,w_5\}.
\]
Their endpoints occur on $C$ in the cyclic order
$t_0,t_2,x_0,x_2$. Since both chords lie outside $D$, they cross,
contradicting planarity.

\begin{figure}[!htbp]
    \centering
    \includegraphics[width=0.3\textwidth]{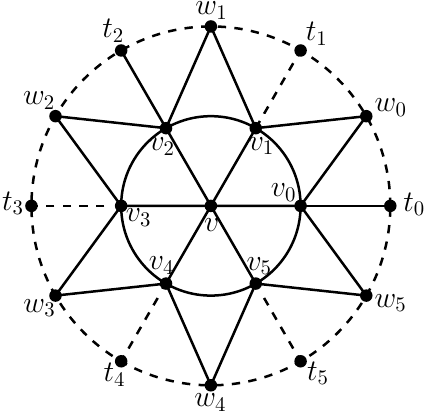}
    \caption{The local configuration around $v$ when
    $d(v_0)=d(v_2)=6$. For $i\notin\{0,2\}$, the vertex $t_i$
    is present precisely when $d(v_i)=6$; otherwise the
    corresponding boundary segment is the edge $w_{i-1}w_i$.}
    \label{fig:degree-six-0-2}
\end{figure}

\smallskip
\noindent\emph{Two degree $6$ neighbours are consecutive.}
After cyclic relabelling, suppose that $d(v_0)=d(v_1)=6$.
By the previous case, $v_2,v_3,v_4,v_5$ all have degree $5$.
Thus \eqref{eq:degree-six-reduced-chords} forces chords
$t_0x_0,t_1x_1\in E(G)$ with
\[
x_0\in\{w_2,w_3\},
\qquad
x_1\in\{w_3,w_4\}.
\]
The edges $t_0w_3$ and $t_1w_3$ would cut off the three-edge
paths $t_0w_5w_4w_3$ and $t_1w_1w_2w_3$, respectively, giving us separating cycles of length $4$.
Hence $x_0=w_2$ and $x_1=w_4$. But $t_0w_2$ and $t_1w_4$
have interlaced endpoints on $C$, again contradicting planarity.

\medskip

We have excluded exactly one degree $6$ neighbour, as well as
any pair at cyclic distance $1$ or $2$. Therefore either all
six neighbours of $v$ have degree $5$, or exactly two have
degree $6$ and they are antipodal.
\end{proof}

We now study the arrangement of degrees around a vertex of degree $5$. But before that we deduce a consequence of the preceeding proposition:
\begin{lemma}
\label{cl:degree-six-continuation}
Let $x$ be a vertex of degree $6$ with two degree $6$ neighbours
$y,z$, and let $y',z'$ be their respective degree $6$ neighbours
other than $x$. Then $y'z'\in E(G)$.
\end{lemma}

\begin{proof}
By the previous proposition, $y$ and $z$ are antipodal around $x$.
We assume $x=v$,
$y=v_0$ and $z=v_3$, as illustrated in Figure \ref{fig:bad}. The remaining neighbours
$v_1,v_2,v_4,v_5$ of $x$ all have degree $5$.

Since $x$ and $v_0$ both have degree $6$, the first assertion
applied at $v_0$ shows that its other degree $6$ neighbour $y'$
is antipodal to $x$. The cyclic order around $v_0$ is
$x,v_5,w_5,t_0,w_0,v_1$, so $y'=t_0$.
Similarly, the cyclic order around $v_3$ is
$x,v_2,w_2,t_3,w_3,v_4$, and hence $z'=t_3$.
It remains to prove that $t_0t_3\in E(G)$.

\begin{figure}[!htbp]
    \centering
    \includegraphics[width=0.3\textwidth]{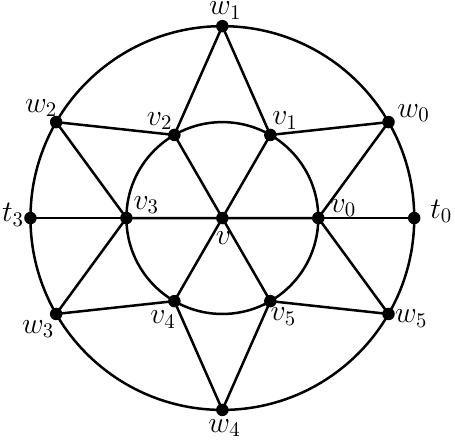}
    \caption{The local configuration around a vertex of degree $6$
    with two antipodal neighbours of degree $6$.}
    \label{fig:bad}
\end{figure}

Since only $v_0$ and $v_3$ have degree $6$ among the neighbours
of $x$, the outer cycle of the local configuration is
\[
C=t_0w_0w_1w_2t_3w_3w_4w_5t_0.
\]
By \eqref{eq:degree-six-reduced-chords}, $t_0$ has a neighbour
in $\{w_2,t_3,w_3\}$. Recall that the short-chord restriction
established in the preceding proof forbids chords joining
vertices at cyclic distance $2$ or $3$ in $C$.

Now $t_0$ and $w_2$ are joined along $C$ by the three-edge path
$t_0w_0w_1w_2$, while $t_0$ and $w_3$ are joined by the
three-edge path $t_0w_5w_4w_3$. Consequently, neither
$t_0w_2$ nor $t_0w_3$ can be an edge.
The only remaining possibility is $t_0t_3\in E(G)$,
which gives $y'z'\in E(G)$.
\end{proof}
\begin{proposition}\label{prop:local-configurations5}
    Let $v$ be a vertex of degree $5$. Then, it has at most 2 neighbors of degree $6$.
\end{proposition}

\begin{proof}
Suppose, for a contradiction, that a vertex $v$ of degree $5$
has at least three neighbours of degree $6$. We use the same
notation $v_i,w_i,t_i$, now with indices taken modulo $5$;
see Figure~\ref{fig:degree-five}.

\begin{figure}[htb]
    \centering
    \includegraphics[width=0.3\textwidth]{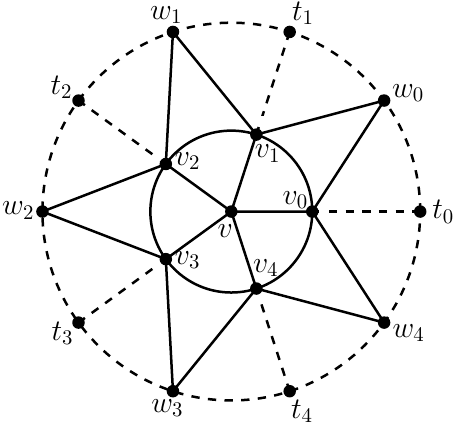}
    \caption{The local configuration around $v$ when
    $d(v)=5$. For each $i$, the vertex $t_i$
    is present precisely when $d(v_i)=6$; otherwise the
    corresponding boundary segment is the edge $w_{i-1}w_i$.}
    \label{fig:degree-five}
\end{figure}

Two degree $6$ neighbours must be consecutive; after cyclic
relabelling, assume that $d(v_0)=d(v_1)=6$.
The result of \ref{prop:local-configurations6} forces $d(v_2)=d(v_4)=5$, since otherwise
$v_1$ or $v_0$ would have two degree $6$ neighbours at cyclic
distance $2$. Thus the only remaining pattern is
\[
d(v_0)=d(v_1)=d(v_3)=6,
\qquad
d(v_2)=d(v_4)=5.
\]

This is illustrated in Figure \ref{fig:5-6-6-6-base}. By antipodality at $v_0$ and $v_1$, we have $d(w_4)=d(w_1)=6$.
On the other hand, the cyclic order around $v_3$ is
$v,v_2,w_2,t_3,w_3,v_4$. Each antipodal pair contains one of
$v,v_2,v_4$, all of degree $5$. Hence the first assertion implies
that every neighbour of $v_3$ has degree $5$.

\begin{figure}[!htbp]
    \centering
    \includegraphics[width=0.3\textwidth]{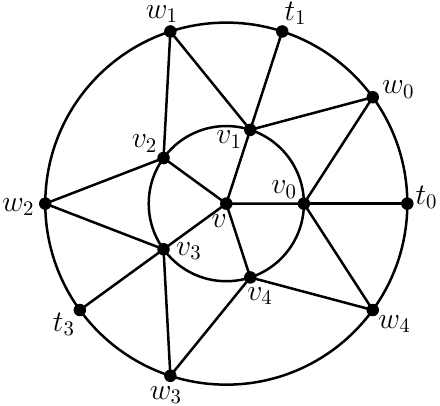}
    \caption{The local configuration around $v$ in the only
    remaining cyclic pattern $6,6,5,6,5$. }
    \label{fig:5-6-6-6-base}
\end{figure}

Let $s$ be the degree $6$ neighbour of $w_1$ other than $v_1$.
Applying Lemma~\ref{cl:degree-six-continuation} with
$(x,y,z)=(v_1,v_0,w_1)$ gives $w_4s\in E(G)$, since the other
degree $6$ neighbour of $v_0$ is $w_4$.
Since $d(v_2)=d(v_4)=5$, the triples $v_1,v_2,w_2$ and
$v_0,v_4,w_3$ are consecutive in the cyclic orders around $w_1$
and $w_4$, respectively. Moreover, $s$ is antipodal to $v_1$
around $w_1$ and to $v_0$ around $w_4$. It follows that
$sw_2,sw_3\in E(G)$.

However, $w_2$ and $w_3$ have cyclic distance $2$ around $v_3$,
so Observation~\ref{obs:local-structure} yields
\[
N(w_2)\cap N(w_3)=\{v_3,t_3\}.
\]
Thus $s\in\{v_3,t_3\}$. This is impossible: $d(t_3)=5$ whereas
$d(s)=6$, and $s\neq v_3$ because $s$ is adjacent to the
degree $6$ vertex $w_1$, while every neighbour of $v_3$ has
degree $5$. This contradiction proves the second assertion.
\end{proof}
We introduce the following terminology:
We call a vertex of degree $6$ \emph{good} if all its neighbours
have degree $5$, and \emph{bad} otherwise.

\subsection{Property $\mathcal{P}_3$: the bound on the number of vertices}

Recall that Euler's formula gives
$|V(G)|-|E(G)|+|F(G)|=2$ for a connected plane graph $G$.
If $G$ is a plane triangulation, then $3|F(G)|=2|E(G)|$, and hence
\begin{equation}
\label{eq:euler-charge}
\sum_{u\in V(G)}\bigl(6-d(u)\bigr)=12.
\end{equation}
As every vertex of our graphs has degree either $5$ or $6$, vertices of degree $6$ do not contribute to the sum, and this forces the graph to have exactly 12 vertices of degree $5$.
We shall now use a discharging argument in order to bound $\vert V(G)\vert$.
\begin{proposition}
\label{prop:vertex-bound}
Let $G$ be a 5-connected plane triangulation with positive
LLY curvature. Then $|V(G)|\leq 17$.
\end{proposition}

\begin{proof} Let $G$ be such a 5-connected triangulation.
Recall that every vertex has degree $5$ or $6$.
We assign to each vertex $u$ the initial charge
$\omega(u):=6-d(u)$. Thus vertices of degree $5$ have charge $1$,
and vertices of degree $6$ have charge $0$.
By \eqref{eq:euler-charge}, the total charge is $12$.
In particular, $G$ has exactly twelve vertices of degree $5$.

\smallskip

We redistribute the charge according to the following rule.
Each vertex of degree $5$ sends charge $1/9$ to each of its
neighbours of degree $6$. If it has two consecutive neighbours
of degree $6$, it sends an additional charge $1/18$ to each of
them. These two vertices are necessarily bad, since they are
adjacent.

\smallskip

Let $\omega'(u)$ denote the final charge of $u$. Since the
discharging procedure only redistributes charge,
\[
\sum_{u\in V(G)}\omega'(u)=12.
\]
Note that for every vertex we have that $\omega'(u)\geq \frac{2}{3}$.

\smallskip

Indeed, if $d(v)=5$, then by Proposition \ref{prop:local-configurations5}, it has at most 2 neighbors of degree $6$. Thus,\[
\omega'(u)=
\begin{cases}
1,
& \text{if $u$ has no neighbour of degree $6$},\\[1mm]
8/9,
& \text{if $u$ has exactly one neighbour of degree $6$},\\[1mm]
7/9,
& \text{if $u$ has two non-consecutive neighbours of degree $6$},\\[1mm]
2/3,
& \text{if $u$ has two consecutive neighbours of degree $6$}.
\end{cases}
\]
which is greater than or equal to $\frac{2}{3}$.\\

Now suppose that $d(u)=6$. If $u$ is good, each of its six
degree $5$ neighbours sends charge $1/9$ to $u$, so
$\omega'(u)=6/9=2/3$.
If $u$ is bad, Proposition~\ref{prop:local-configurations6}
gives exactly two antipodal neighbours of degree $6$.
Each of its four degree $5$ neighbours is incident with a
triangular face containing $u$ and one of these two vertices.
It therefore has two consecutive neighbours of degree $6$
and sends charge $1/9+1/18=1/6$ to $u$. Consequently,
$\omega'(u)=4/6=2/3$.

Thus,
\[
12=\sum_{u\in V(G)}\omega'(u)\geq \frac23\cdot |V(G)|,
\]
and hence $|V(G)|\leq 18$.

\medskip

It remains to exclude equality. Suppose, for a contradiction,
that $|V(G)|=18$. This can only happen if every vertex has final charge exactly
$2/3$. Thus, every vertex of degree $5$
has two consecutive neighbours of degree $6$, both of which
are bad. There can therefore be no good vertex of degree $6$:
any of its degree $5$ neighbours would have a good degree $6$
neighbour, a contradiction.

\smallskip

Since $G$ has exactly twelve vertices of degree $5$, its
remaining six vertices have degree $6$ and are all bad.
Let $D$ be the subgraph induced by these six vertices.
By Proposition~\ref{prop:local-configurations6}, every
vertex of $D$ has degree $2$, so $D$ is a disjoint union of cycles.

Moreover, $D$ contains no triangle. Indeed, every triangle of
$G$ is facial by 5-connectivity, so a triangle of degree $6$
vertices would give each of its vertices two consecutive
degree $6$ neighbours, contrary to their antipodality.
Since $D$ has six vertices, it follows that $D\cong C_6$.
Write this cycle as $a_0a_1a_2a_3a_4a_5a_0$.
Applying Claim~\ref{cl:degree-six-continuation} with
$(x,y,z)=(a_0,a_1,a_5)$ gives $a_2a_4\in E(G)$, since the
other degree $6$ neighbours of $a_1$ and $a_5$ are $a_2$ and
$a_4$, respectively. This contradicts the fact that $D$ is
an induced $6$-cycle.

\medskip

Thus $|V(G)|=18$ is impossible, and $|V(G)|\leq 17$.
\end{proof}
\subsection{Property $\mathcal{P}_4$: the classification}
%\textcolor{blue}{I was wondering whether it would be a good idea to actually calculate the curvature in each of the $G_i$'s explicitely? I'm not really satisfied with how the proof ends. I also think that it's a bit redundant to draw the same graphs twice, once in the beginning of the section and once in the proof of the final property.}

As observed above, if $G$ is a 5-connected plane triangulation with positive LLY curvature, then all of its vertices have degree either $5$ or $6$ and by Euler's formula, we know that it must have exactly 12 vertices of degree $5$. For such a graph, let $n_6$ be the number of degree $6$ vertices. Then, by the above proposition, as $\vert V(G)\vert= 12+n_6 \leq 17 $, we deduce that $n_6 \leq 5$. In order to fully classify such triangulations, we first record the consequences of Euler's formula that will be used
throughout this subsection.

\medskip

We start with an immediate consequence of Euler's formula (obtained by adding a vertex adjacent to all vertices on the boundary of $R$).

\begin{observation}\label{obs:euler}
Suppose that $R$ is a triangulated disk whose boundary is a
cycle of length $a$, and suppose that $R$ contains $r$ interior
vertices. Then the triangulation of $R$ has $2a+3r-3$ edges in total,
of which $a+3r-3$
do not lie on the boundary.
\end{observation}

%\begin{proof}
%The first identity is an immediate consequence of a Euler's formula for connected plane graphs.
%Now let $R$ be a triangulated disk with boundary length $a$ and
%$r$ interior vertices. If $m$ and $f$ denote respectively its
%numbers of edges and triangular faces, then Euler's formula for the
%disk gives $a+r-m+f=1$. Moreover, every interior edge belongs to two
%triangular faces and every boundary edge to one, so $3f=2m-a$.
%Eliminating $f$ gives $m=2a+3r-3$. Removing the $a$ boundary edges
%leaves $a+3r-3$ edges not lying on the boundary.
%\end{proof}

We next isolate the two structural facts needed in the classification.

\begin{lemma}
\label{lem:bad-extremal}
Let $G$ be a $5$-connected plane triangulation with positive LLY curvature. $G$ has a bad vertex if and only if $n_6=5$.
In this case, all vertices of degree $6$ are bad and induce a
$5$-cycle.
\end{lemma}

\begin{proof}
Let $H$ be the subgraph of $G$ induced by the vertices of degree $6$.
By Proposition~\ref{prop:local-configurations6}, every vertex of
degree $6$ is either good, in which case all its neighbours have
degree $5$, or bad, in which case it has exactly two degree $6$
neighbours, and these are antipodal in its cyclic neighbourhood.
Thus every vertex of $H$ has degree $0$ or $2$, and every nontrivial
component of $H$ is a cycle.

Suppose now that $G$ has a bad vertex. By Lemma \ref{cl:degree-six-continuation}, its component in $H$ is a $5$-cycle. Moreover, every other vertex of degree $6$ will have exactly $2$ neighbors of degree $6$, making all of the degree $6$ vertices bad.

Conversely, suppose that $n_6=5$ and that all degree $6$ vertices
are good. Each of them then has six neighbours of degree $5$, so
there are $6n_6=30$ edges between vertices of degrees $5$ and $6$.
On the other hand, by Observation \ref{obs:euler} there are exactly
twelve vertices of degree $5$, and
Proposition~\ref{prop:local-configurations5} shows that each has
at most two neighbours of degree $6$. Hence
$6n_6\leq2\cdot12=24$, a contradiction. Therefore $G$ has a bad
vertex, and the first part of the proof applies.
\end{proof}

\begin{lemma}
\label{lem:disk-completion}
Let $R$ be a disk bounded by a cycle $C$ of length
$a\in\{5,6\}$, with at least one vertex of $G$ outside $C$. Suppose that every vertex of $C$ is incident with at most
four edges drawn on $C$ or outside. Then $C$ has no chord
in $R$, and $R$ contains at least one interior vertex.

If $R$ contains $r$ interior vertices, then its triangulation has
$a+3r-3$ edges not lying on $C$. In particular, if $r=1$, the
unique interior vertex is adjacent to every vertex of $C$, and the
triangulation of $R$ is completely determined.
\end{lemma}

\begin{proof}
Suppose first that $xy$ is a chord of $C$ contained in $R$.
The two $x$--$y$ paths along $C$ have total length $a\leq6$, so the
shorter one has length at most $3$. Since $xy$ is a chord, this path
has length at least $2$. Hence $xy$ together with the shorter path
bounds in $R$ either a triangle or a quadrilateral.

This region contains no interior vertex. Indeed, otherwise its
boundary, consisting of at most four vertices, would separate such
a vertex from the vertex of $G$ lying on the other side of $C$,
contradicting $5$-connectivity.

If the shorter path has length 2, say $xuy$, then the region is the triangular face
$xuy$, and $u$ has no incident edge entering the interior of $R$.
If the shorter path is $xuvy$, the empty quadrilateral is
triangulated by a single diagonal, so at least one of $u$ and $v$
has no incident edge entering the interior of $R$. In either case,
some internal vertex of the chosen $x$--$y$ path is incident only
with edges lying on $C$ or on the other side of $C$. By assumption,
it is incident with at most four such edges, and hence has degree at
most $4$, contradicting $\delta(G)\geq5$. Therefore $C$ has no chord
in $R$.

\smallskip

Let $r$ be the number of vertices in the interior of $R$. By
Observation~\ref{obs:euler}, the triangulation of $R$ has
$a+3r-3$ edges not lying on $C$.

\smallskip

If $r=0$, all these $a-3$ edges have both endpoints on $C$ and are
therefore chords of $C$. Since $a\in\{5,6\}$, we have $a-3>0$,
contradicting the first part of the proof. Thus $r\geq1$.

Finally, suppose that $r=1$, and let $p$ be the unique interior
vertex. There are exactly $a$ edges not lying on $C$. Since $C$ has
no chord, every one of them joins $p$ to a vertex of $C$. As $C$
has exactly $a$ vertices, $p$ is adjacent to every vertex of $C$.
Thus the triangulation of $R$ is completely determined.
\end{proof}

We can now prove the classification.

\begin{proposition}
\label{prop:classification}
Let $G$ be a $5$-connected plane triangulation with positive LLY curvature. Then
\[
    G\cong G_{12},\,G_{14},\,G_{15},\,G_{16},
    \quad\text{or}\quad G_{17}.
\]
%Conversely, each of these five graphs is a $5$-connected plane
%triangulation with strictly positive Lin--Lu--Yau curvature on every
%edge.
\end{proposition}

\begin{proof}
For the cases $n_6\leq3$, we use the following notation. If $n_6=0$,
choose a vertex $v$ of degree $5$; otherwise choose a vertex $v$ of
degree $6$. Put $a:=d(v)$. When $n_6>0$,
Lemma~\ref{lem:bad-extremal} shows that $v$ is good, so all its
neighbours have degree $5$. This also holds when $n_6=0$.

\smallskip

Write the neighbours of $v$ in cyclic order as
$v_0,\ldots,v_{a-1}$, with indices taken modulo $a$, and let $w_i$
be the common neighbour of $v_i$ and $v_{i+1}$ different from $v$.
By Observation~\ref{obs:local-structure}, the vertices
$w_0,\ldots,w_{a-1}$ lie outside $N[v]$ and are pairwise distinct.
Since $d(v_i)=5$,
$N(v_i)=\{v,v_{i-1},w_{i-1},w_i,v_{i+1}\}$ in cyclic order.
Hence $w_{i-1}w_i\in E(G)$ for every $i$, and
$W=w_0w_1\cdots w_{a-1}w_0$ is a cycle.

\smallskip

The faces incident with the vertices of $N[v]$ triangulate the disk
bounded by $W$ and containing $v$. Each $w_i$ is incident with
exactly four edges on this side of $W$, namely
$w_iv_i,w_iv_{i+1},w_iw_{i-1},w_iw_{i+1}$. All remaining vertices
lie in the disk $R$ on the other side of $W$, and their number is
$r=|V(G)|-(1+2a)$. Lemma~\ref{lem:disk-completion} applies to $R$;
in particular, $W$ has no chord in $R$ and $r\geq1$.

\smallskip
\noindent\emph{Case 1: $n_6=0$.}

\noindent Here $|V(G)|=12$ and $a=5$, so $r=1$. By
Lemma~\ref{lem:disk-completion}, the unique remaining vertex $p$ is
adjacent to all five vertices of $W$. The graph is therefore
completely determined and is the icosahedron $G_{12}$.

\begin{figure}[H]
    \centering
    \includegraphics[width=0.25\textwidth]
    {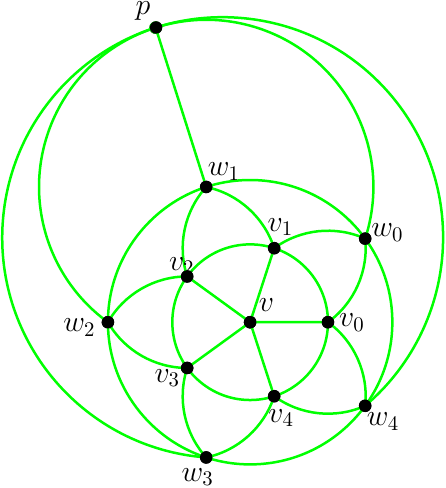}
    \caption{The local configuration used in the construction
    of the icosahedron $G_{12}$.}
    \label{fig:icosahedron-construction}
\end{figure}

\smallskip
\noindent\emph{Case 2: $n_6=1$.}

\noindent Here $|V(G)|=13$ and $a=6$, so $r=0$, contrary to
Lemma~\ref{lem:disk-completion}. Hence no such graph exists.

\smallskip
\noindent\emph{Case 3: $n_6=2$.}

\noindent Here $|V(G)|=14$ and $a=6$, so $r=1$. By
Lemma~\ref{lem:disk-completion}, the unique remaining vertex $p$ is
adjacent to all six vertices of $W$, and hence $d(p)=6$. Each $w_i$
then has precisely the five neighbours
$v_i,v_{i+1},w_{i-1},w_{i+1},p$, and therefore has degree $5$.
The graph is completely determined. This is $G_{14}$.

\begin{figure}[H]
    \centering
    \includegraphics[width=0.25\textwidth]{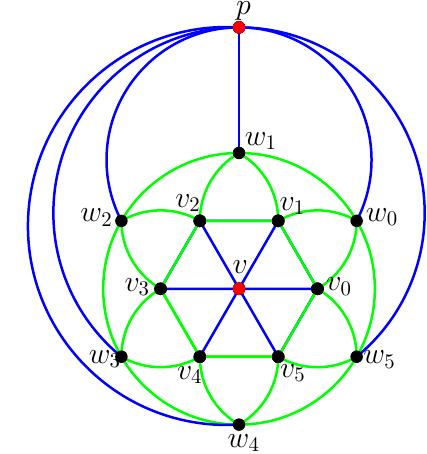}
    \caption{The configuration in the case $n_6=2$, giving
    $G_{14}$.}
    \label{fig:n6-equals-2}
\end{figure}

\smallskip
\noindent\emph{Case 4: $n_6=3$.}

\noindent Here $|V(G)|=15$ and $a=6$, so $r=2$. Let $p,q$ be the two vertices
in the interior of $R$, let $s$ be the number of edges joining
$\{p,q\}$ to $W$, and let $\varepsilon=1$ if $pq\in E(G)$ and
$\varepsilon=0$ otherwise.

\smallskip

Since $W$ has no chord in $R$, every edge of the triangulation of
$R$ not lying on $W$ either joins $\{p,q\}$ to $W$ or is the edge
$pq$. By Observation~\ref{obs:euler}, there are
$6+3\cdot2-3=9$ such edges. Hence $s+\varepsilon=9$. Moreover,
$d(p)+d(q)=s+2\varepsilon\geq10$, as $\delta(G) =5$. It follows that
$\varepsilon=1$, $s=8$, and $d(p)=d(q)=5$.

\smallskip

The edge $pq$ is incident with two triangular faces, say $pqw_i$
and $pqw_j$. The neighbours of $p$ on $W$ form one boundary interval
between $w_i$ and $w_j$, while those of $q$ form the complementary
interval. Since each of $p,q$ has degree $5$, each has exactly four
neighbours on $W$. Thus both intervals contain four vertices, and
$w_i,w_j$ are opposite on the $6$-cycle $W$. After cyclic
relabelling,
\[
    p\sim w_0,w_1,w_2,w_3,\qquad
    q\sim w_3,w_4,w_5,w_0,\qquad pq\in E(G).
\]
Thus $w_0$ and $w_3$ have degree $6$, while the other four vertices
of $W$ have degree $5$. The three degree $6$ vertices are precisely
$v,w_0,w_3$, and the graph is completely determined. This is
$G_{15}$.

\begin{figure}[!htbp]
    \centering
    \includegraphics[width=0.3\textwidth]{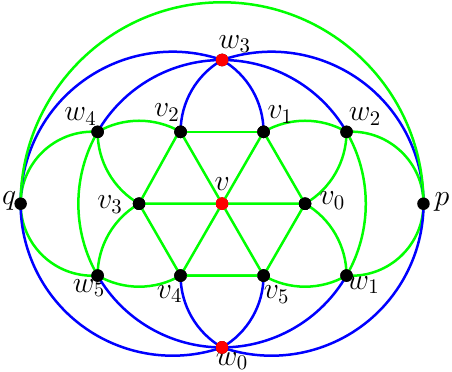}
    \caption{The configuration in the case $n_6=3$, giving $G_{15}$. }
    \label{fig:n6-equals-3}
\end{figure}

\smallskip
\noindent\emph{Case 5: $n_6=4$.}

\noindent Here $|V(G)|=16$. By Lemma~\ref{lem:bad-extremal}, all four
vertices of degree $6$ are good. Let $S$ be their set and put
$T:=V(G)\setminus S$. Thus $|T|=12$, every vertex of $T$ has
degree $5$, and $S$ is independent.

Each vertex of $S$ has six neighbours in $T$, so there are
$4\cdot6=24$ edges between $S$ and $T$. On the other hand,
Proposition~\ref{prop:local-configurations5} says that each
vertex of $T$ has at most two neighbours in $S$. The twelve
vertices of $T$ can therefore be incident with at most
$12\cdot2=24$ edges to $S$. Since this bound is attained,
every vertex of $T$ has exactly two neighbours in $S$.

Let $H:=G[T]$. Since $G$ is $5$-connected and $|S|=4$, the
graph $H=G-S$ is connected. Each vertex of $T$ has degree
$5$ in $G$, of which two neighbours lie in $S$. Hence every
vertex has degree $3$ in $H$. It follows that
$|E(H)|=3\cdot12/2=18$, and Euler's formula gives
$|F(H)|=2-12+18=8$.

\begin{figure}[!htbp]
    \centering
    \includegraphics[width=0.25\textwidth]{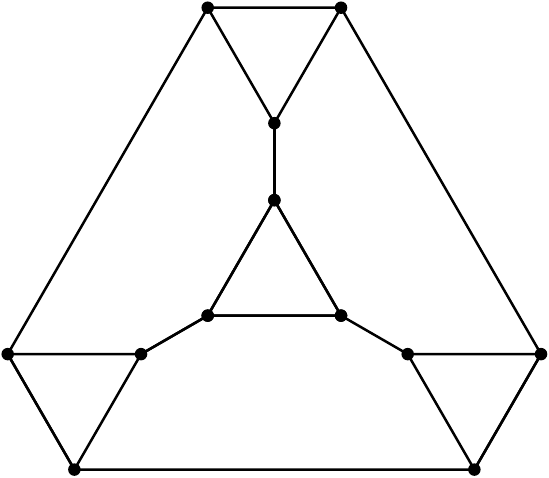}
    \caption{The $3$-regular graph $H=G[T]$ induced by the twelve
    vertices of degree $5$ in the case $n_6=4$.}
    \label{fig:H-n16}
\end{figure}

For each $a\in S$, its six neighbours form a cycle. The six
triangular faces incident with $a$ fill one side of this
cycle and, when $a$ is removed, merge into a hexagonal face.
No other vertex lies in that region. Moreover, since $S$ is
independent, the cycle consists entirely of vertices and
edges of $H$. The four regions therefore remain four
distinct hexagonal faces after all of $S$ is removed.

The sum of the face lengths of $H$ is $2|E(H)|=36$. Its four
hexagonal faces contribute $4\cdot6=24$, leaving total
length $12$ for the other four faces. Each has length at
least $3$, so all four are triangles.

We claim that these four triangular faces are pairwise
vertex-disjoint. Fix $v\in T$. Since $d_G(v)=5$, five
triangular faces of $G$ are incident with $v$. Each of its
two edges to $S$ lies in two of these faces. The resulting
four faces are distinct: a face containing both neighbours
of $v$ in $S$ would make them adjacent, contrary to the
independence of $S$. Thus exactly one triangular face
incident with $v$ avoids $S$, and it remains a triangular
face of $H$. Every vertex of $H$ therefore lies on exactly
one of its four triangular faces. These faces are
vertex-disjoint and partition $V(H)$.

\smallskip

Replace each of the four triangles by a single vertex,
keeping the edges that join different triangles, and call
the resulting graph $Q$. Each vertex of a triangle has
exactly one edge leaving that triangle. Thus $Q$ is
connected, has four vertices, and each of its vertices has
three incident edges.

We show that two triangles cannot be joined by two edges.
If they were, those two edges, together with one edge from
each triangle, would form a $4$-cycle in $H$. The cycle has
no chord: each of its four vertices has already used its
unique edge leaving its triangle. As $H$ is induced, it has
no chord in $G$ either. A chordless $4$-cycle in a plane
triangulation has vertices on both sides, since a side
without vertices would need a diagonal to be triangulated.
It would therefore separate $G$, contradicting
$5$-connectivity.

Consequently, each triangle is joined by one edge to each
of the other three triangles. In other words, $Q\cong K_4$.
Furthermore, the three edges leaving any triangle attach
to its three different vertices. This determines $H$ up
to isomorphism: it is obtained from $K_4$ by replacing each
vertex with a triangle.

Finally, each of the four hexagonal faces of $H$ contains
one vertex of $S$, adjacent to all six vertices on its
boundary. Hence $G$ is uniquely determined and is
isomorphic to $G_{16}$.

\begin{figure}[H]
    \centering
    \includegraphics[width=0.25\textwidth]{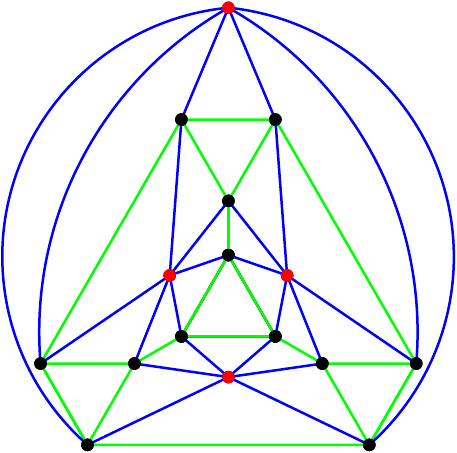}
    \caption{The graph $G_{16}$. The black vertices induce $H$.
    The four red vertices lie in its hexagonal faces and are
    adjacent to all six vertices on the corresponding boundary. }
    \label{fig:n16}
\end{figure}

\smallskip

\noindent\emph{Case 6: $n_6=5$.}

\noindent Here $|V(G)|=17$. By Lemma~\ref{lem:bad-extremal}, the degree $6$
vertices are all bad and induce a $5$-cycle
$A=a_0a_1a_2a_3a_4a_0$, with indices taken modulo $5$.

For each edge $a_ia_{i+1}$, let $b_i$ and $c_i$ be the third
vertices of its two incident faces, chosen on the two sides of $A$.
All these vertices have degree $5$.

They are pairwise distinct. Indeed, if the same vertex occurred for
two distinct edges of $A$, it would be adjacent to at least three
vertices of $A$, hence to at least three degree $6$ vertices,
contrary to Proposition~\ref{prop:local-configurations5}.

By antipodality, the cyclic order around $a_i$ is, up to reversing
the orientation,
\[a_{i-1},b_{i-1},b_i,a_{i+1},c_i,c_{i-1}\]
Hence $b_{i-1}b_i,c_{i-1}c_i\in E(G)$ for every $i$.
Thus the vertices $b_i$ form a $5$-cycle $B$ on one side of $A$,
and the vertices $c_i$ form a $5$-cycle $C$ on the other side.

The three cycles $A,B,C$ account for fifteen vertices. Each $b_i$
is incident, on the side of $B$ containing $A$, with exactly the
four edges
$b_ia_i,b_ia_{i+1},b_ib_{i-1},b_ib_{i+1}$, and similarly for the
$c_i$. Lemma~\ref{lem:disk-completion} therefore applies to the disk
bounded by $B$ on the side not containing $A$, and also to the
corresponding disk bounded by $C$. Each contains at least one
interior vertex.

Exactly two vertices remain, so each disk contains exactly one,
say $p$ inside $B$ and $q$ inside $C$. By
Lemma~\ref{lem:disk-completion},
$N(p)=\{b_0,b_1,b_2,b_3,b_4\}$ and
$N(q)=\{c_0,c_1,c_2,c_3,c_4\}$.
All neighbourhoods are now determined, giving $G_{17}$.

\begin{figure}[htb]
    \centering
    \includegraphics[width=0.25\textwidth]{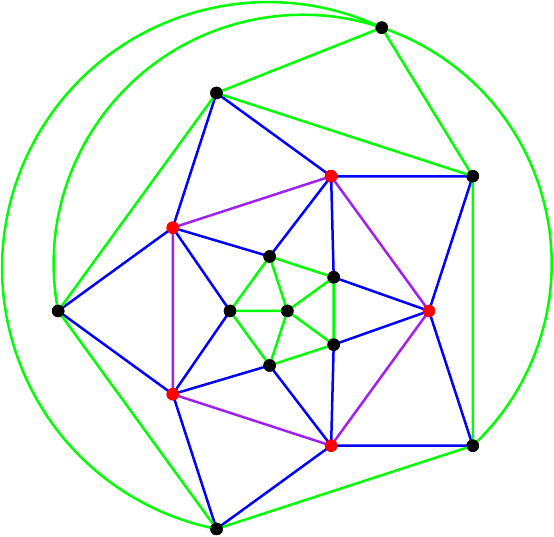}
    \caption{The configuration in the case $n_6=5$.
    The five degree $6$ vertices form $A$.
    The two remaining vertices are adjacent to all vertices
    of $B$ and $C$, respectively.}
    \label{fig:n17-alt}
\end{figure}
The six cases exhaust all possibilities, and in each admissible case
the graph is uniquely determined. Hence
$G\cong G_n$ for some $n\in\{12,14,15,16,17\}$.
\end{proof}

\subsection*{AI disclosure} GPT-4 was used in March 2026
to help produce the python code for computing the LLY curvature in
graphs. Even if this is not directly visible in the manuscript, the computational
experiments we did in python with a large number of planar neighborhoods were crucial
to guide us towards the final proof of Theorem \ref{thm:5ct}. GPT-5.5 was used in
July 2026 to find a version of
Lemma \ref{lem:common-intersection} and trace it back to \cite{KovariSosTuran1954}. 

\subsection*{Ackowledgements} The authors would like to thank Zixia Song, Baoxuan Zhu and Xiaonan Liu for the interesting discussions.

\bibliographystyle{alpha}
\bibliography{ref}

\appendix

\section{Proof of Lemma \ref{lem:common-intersection}}\label{sec:app}

If $r=1$, the conclusion is immediate. Assume $r\ge2$.

Construct a bipartite graph $H$ with parts $X=[N]$ and $Y=[q]$,
viewed as disjoint copies, by joining $a\in X$ to $i\in Y$
whenever $a\in S_i$. Thus $N_H(i)=S_i$, and counting edges gives
\[
\sum_{a\in X}d_H(a)=\sum_{i\in Y}|S_i|\ge qm.
\]

We double-count the pairs $(a,I)$ where $a\in X$ and
$I$ is an $r$-element subset of $N_H(a)$.
First, fixing $a\in X$, there are exactly
$\binom{d_H(a)}{r}$ choices for $I$.
Conversely, fixing an $r$-element subset $I\subseteq Y$,
the possible vertices $a$ are precisely the common neighbours
of the vertices in $I$, namely the elements of
$\bigcap_{i\in I}S_i$. Therefore,
\[
\sum_{a\in X}\binom{d_H(a)}{r}
=
\sum_{\substack{I\subseteq Y\\|I|=r}}
\left|\bigcap_{i\in I}S_i\right|.
\]

%To estimate the left-hand side, define
%$\varphi:[0,\infty)\to[0,\infty)$ by linear interpolation
%between consecutive values $\varphi(k)=\binom{k}{r}$:
%\[
%\varphi(k+\theta)
%=
%(1-\theta)\binom{k}{r}
%+\theta\binom{k+1}{r},
%\qquad
%k\in\mathbb N,\quad \theta\in[0,1],
%\]
%where $\binom{k}{r}=0$ for $k<r$.
%Its slope on $[k,k+1]$ is
%\[
%\binom{k+1}{r}-\binom{k}{r}=\binom{k}{r-1}.
%\]
%These slopes are nonnegative and nondecreasing, so
%$\varphi$ is nondecreasing and convex.
%
%Recall that Jensen's inequality for a convex function states that
%\[
%\frac1N\sum_{a\in X}\varphi(u_a)
%\ge
%\varphi\left(\frac1N\sum_{a\in X}u_a\right).
%\]
By Jensen's inequality,
\[
\sum_{a\in X}\binom{d_H(a)}{r}
\ge N\cdot \binom{\frac1N\sum_{a\in X}d_H(a)}{r}
\ge N\cdot \binom{qm/N}{r}.
\ge N\cdot \left(\frac{qm}{Nr}\right)^r.
\]

%Since $qm/N\ge r\ge2$, we have
%$\lfloor qm/N\rfloor\ge r$ and
%$\lfloor qm/N\rfloor\ge qm/(2N)$.
%Moreover, for integers $k\ge r$,
%\[
%\binom{k}{r}
%=
%\prod_{j=0}^{r-1}\frac{k-j}{r-j}
%\ge\left(\frac{k}{r}\right)^r.
%\]
%Consequently,
%\[
%\sum_{a\in X}\binom{d_H(a)}{r}
%\ge N\left(\frac{qm}{2rN}\right)^r.
%\]

Since $m\le N$, the hypothesis implies $q\ge r$.
By the double-counting identity and
$\binom{q}{r}\le(eq/r)^r$, the average size of an
$r$-fold intersection is at least
\[
\frac{1}{\binom{q}{r}}
\sum_{\substack{I\subseteq Y\\|I|=r}}
\left|\bigcap_{i\in I}S_i\right|
\ge
\frac{N\left(\dfrac{qm}{rN}\right)^r}
     {\left(\dfrac{eq}{r}\right)^r}
=
N\cdot \left(\frac{m}{eN}\right)^r.
\]
Hence some distinct indices $i_1,\ldots,i_r$ satisfy
the required bound.
\qed

\end{document}